%% file: main.tex
\documentclass[a4paper,11pt]{article}
\usepackage[T1]{fontenc}
\usepackage[utf8]{inputenc}
\usepackage{lmodern}
\usepackage{bm}
\usepackage{geometry}
\usepackage{amsthm}
\usepackage{mathtools}
\usepackage[backend=bibtex,style=numeric,sorting=none]{biblatex}
\title{Adjoint-Compat ible Surrogates of the Expected Information Gain for Optimal Experimental Design}
\author{
Luc de Montella\thanks{Otto von Guericke University Magdeburg, Magdeburg, Germany. 
Email: luc.demontella@ovgu.de}
\and
Sebastian Sager\thanks{Otto von Guericke University Magdeburg, Magdeburg, Germany. 
Max Planck Institute for Dynamics of Complex Technical Systems, Magdeburg, Germany. 
Email: sager@ovgu.de}
}
\date{}

\usepackage{amsmath,amssymb,amsfonts}

\newtheorem{proposition}{Proposition}

\newtheorem{lemma}{Lemma}
\newtheorem{remark}{Remark}

\usepackage[colorlinks=true,linkcolor=blue,urlcolor=blue,citecolor=blue]{hyperref}

\newcommand{\E}{\mathbb{E}}

\usepackage{xparse}

\providecommand{\email}[1]{\texttt{#1}}

\makeatletter
\providecommand{\@keywords}{}
\providecommand{\@mscodes}{}

\@ifundefined{keywords}{
  \NewDocumentEnvironment{keywords}{+b}
    {\gdef\@keywords{#1}}
    {}
}{}

\@ifundefined{MSCcodes}{
  \NewDocumentEnvironment{MSCcodes}{+b}
    {\gdef\@mscodes{#1}}
    {}
}{}
\makeatother
\usepackage{booktabs}
\usepackage{booktabs}
\usepackage{multirow}
\input{utils/plots}

\usepackage{pgfplotstable}
\pgfplotsset{compat=1.18}
 
\pgfplotstableread[col sep=comma]{csv/harmonic_similar_eig.csv}\harmsim
\pgfplotstableread[col sep=comma]{csv/harmonic_uneven_eig.csv}\harmunev
\pgfplotstableread[col sep=comma]{csv/lv_unimodal_eig.csv}\lvuni
\pgfplotstableread[col sep=comma]{csv/lv_bimodal_eig.csv}\lvbi
 
\newcommand{\V}[3]{%
  \pgfplotstablegetelem{#2}{#3}\of#1%
  \pgfplotsretval
}
\newcommand{\Ve}[2]{%
  \pgfplotstablegetelem{#2}{best}\of#1%
  \edef\tmpbest{\pgfplotsretval}%
  \pgfplotstablegetelem{#2}{EIG}\of#1%
  \ifnum\tmpbest=1\relax
    \textbf{\pgfplotsretval}%
  \else
    \pgfplotsretval
  \fi
}

\pgfplotstableread[col sep=comma]{csv/scaling_N_inst.csv}\Ninst
\pgfplotstableread[col sep=comma]{csv/scaling_N_tilt.csv}\Ntilt
\pgfplotstableread[col sep=comma]{csv/scaling_Q_inst.csv}\Qinst
\pgfplotstableread[col sep=comma]{csv/scaling_Q_tilt.csv}\Qtilt
 
\pgfplotsset{
  scaling common/.style={
    grid=major,
    grid style={line width=0.2pt, gray!25},
    tick label style={font=\footnotesize},
    label style={font=\small},
    title style={font=\small\bfseries},
    width=5.8cm, height=5cm,
    legend style={font=\scriptsize, at={(0.03,0.97)}, anchor=north west,
                  draw=none, fill=none},
  },
  inst/.style={blue!80!black, thick, mark=o, mark size=1.8pt},
  tilt/.style={orange!80!black, thick, mark=triangle*, mark size=2.2pt},
  inst dashed/.style={inst, dashed, thin, opacity=0.6},
  tilt dashed/.style={tilt, dashed, thin, opacity=0.6},
}

\title{Adjoint-Compatible Surrogates of the Expected Information Gain for Optimal Experimental Design}

\author{
Luc de Montella\thanks{Otto von Guericke University Magdeburg, Magdeburg, Germany
(\email{luc.demontella@ovgu.de}, \email{sager@ovgu.de}).}
\and
Sebastian Sager\footnotemark[1]\thanks{Max Planck-Institute for Dynamics of Complex Technical Systems, Magdeburg, Germany}
}

\usepackage{algorithm}
\usepackage{algorithmic}

\begin{document}

\maketitle

\begin{abstract}
We consider optimal experimental design for parameter estimation in dynamical systems governed by controlled ordinary differential equations. In such problems, Fisher-based criteria are attractive because they lead to time-additive objectives compatible with adjoint-based optimal control, but they remain intrinsically local and may perform poorly under strong nonlinearities or non-Gaussian prior uncertainty. By contrast, the expected information gain (EIG) provides a principled Bayesian objective, yet it is typically too costly to evaluate and does not naturally admit an adjoint-compatible formulation.
In this work, we introduce adjoint-compatible surrogates of the EIG based on an exact chain-rule decomposition and tractable approximations of the posterior distribution of the unknown parameter. This leads to two surrogate criteria: an instantaneous surrogate, obtained by replacing the posterior with the prior, and a Gaussian tilting surrogate, obtained by reweighting the prior through a design-driven quadratic information factor. 
We establish theoretical properties of these surrogates, including exactness of the Gaussian tilting surrogate in the linear-Gaussian setting, and illustrate their behavior on benchmark controlled dynamical systems. The results show that the proposed surrogates remain competitive in nearly Gaussian regimes and provide clearer benefits over Fisher-based designs when prior uncertainty is non-Gaussian or multimodal.
\end{abstract}

%

\makeatletter
\ifx\@keywords\@empty\else
\par\smallskip
\noindent\textbf{Keywords.} \@keywords\par
\fi

\ifx\@mscodes\@empty\else
\smallskip
\noindent\textbf{MSC codes.} \@mscodes\par
\fi
\makeatother

\section{Introduction}

\input{sections/introduction}

\section{Problem Formulation}\label{sec:problem_formulation}

\input{sections/general_settings}

\section{Adjoint-compatible surrogates of the expected information gain}\label{sec:surrogates}

\input{sections/eig_approximations}

\section{Error Analysis of the Surrogate Objectives}\label{sec:theory}

\input{sections/theoretical_results}

\section{Adjoint-compatible optimal control formulation}\label{sec:optim_problem_formulation}

\input{sections/in_practice}

\section{Numerical Results}\label{sec:numerical_results}

\input{sections/numerical_results}

\section{Conclusion}

\input{sections/conclusion}

\section*{Acknowledgement}

The authors gratefully acknowledge the funding by the European Regional Development Fund (ERDF) within the programme Research and Innovation - Grant Number ZS/2023/12/182075 and by the Deutsche Forschungsgemeinschaft (DFG, German Research Foundation) via grant 314838170, GRK 2297 MathCoRe.

\appendix

\section{Technical Lemmas}

\input{sections/appendix}

\printbibliography


\end{document}

%% file: utils/plots.tex
\usepackage{xcolor}
\usepackage{pgfplots}
\usepgfplotslibrary{groupplots,statistics}
\usepackage{pgfplotstable}
\pgfplotsset{compat=1.18}

\newcommand{\AddBoxFromRow}[1]{%
  \pgfplotstablegetelem{#1}{whislo}\of{\datatable}\edef\WhisLo{\pgfplotsretval}%
  \pgfplotstablegetelem{#1}{q1}\of{\datatable}\edef\QOne{\pgfplotsretval}%
  \pgfplotstablegetelem{#1}{median}\of{\datatable}\edef\MedianVal{\pgfplotsretval}%
  \pgfplotstablegetelem{#1}{q3}\of{\datatable}\edef\QThree{\pgfplotsretval}%
  \pgfplotstablegetelem{#1}{whishi}\of{\datatable}\edef\WhisHi{\pgfplotsretval}%
  \pgfplotstablegetelem{#1}{mean}\of{\datatable}\edef\MeanVal{\pgfplotsretval}%
  \edef\TempPlot{%
    \noexpand\addplot+[
      black,
      line width=0.5pt,
      boxplot/draw direction=y,
      boxplot prepared={
        lower whisker=\WhisLo,
        lower quartile=\QOne,
        median=\MedianVal,
        upper quartile=\QThree,
        upper whisker=\WhisHi,
        average=\MeanVal
      },
      boxplot/every box/.style={draw=black, fill=none},
      boxplot/every whisker/.style={draw=black, line width=0.5pt},
      boxplot/every median/.style={draw=orange!90!black, line width=1pt},
      boxplot/every average/.style={
        mark=triangle*,
        mark size=2.2pt,
        draw=green!50!black,
        fill=green!50!black
      }
    ] coordinates {};
  }%
  \TempPlot
}

\newcommand{\ErrorBoxplotFigure}[3]{%
  \pgfplotstableread[col sep=comma]{#1}\datatable
  \begin{figure}[h]
  \centering
  \begin{tikzpicture}
  \begin{groupplot}[
      group style={group size=3 by 1, horizontal sep=1.2cm},
      width=0.33\textwidth,
      height=0.38\textwidth,
      ymin=0,
      tick pos=left,
      xtick style={draw=none},
      boxplot/draw direction=y,
      enlarge x limits=0.1,
      grid=none,
      axis line style={black},
      tick style={black},
      xtick={1,2,3,4},
      xticklabels={A-optimal,Linearised,EIG-instantaneous,EIG-tilting},
      x tick label style={rotate=40, anchor=north east, font=\small},
      tick label style={font=\small},
      label style={font=\small},
  ]

  \nextgroupplot[
      title={Overall error},
      ylabel={Relative error (\%)},
  ]
  \AddBoxFromRow{0}
  \AddBoxFromRow{6}
  \AddBoxFromRow{9}
  \AddBoxFromRow{12}

  \nextgroupplot[
      title={Error on first component},
  ]
  \AddBoxFromRow{1}
  \AddBoxFromRow{7}
  \AddBoxFromRow{10}
  \AddBoxFromRow{13}

  \nextgroupplot[
      title={Error on second component},
  ]
  \AddBoxFromRow{2}
  \AddBoxFromRow{8}
  \AddBoxFromRow{11}
  \AddBoxFromRow{14}

  \end{groupplot}
  \end{tikzpicture}
  \caption{#2}
  \label{#3}
  \end{figure}
}

\newcommand{\ErrorBoxplotFigureWithSurrogateZoom}[4]{%
  \pgfplotstableread[col sep=comma]{#1}\datatable
  \begin{figure}[h]
  \centering
  \begin{tikzpicture}
  \begin{groupplot}[
      group style={group size=4 by 1, horizontal sep=0.9cm},
      width=0.255\textwidth,
      height=0.36\textwidth,
      ymin=0,
      tick pos=left,
      xtick style={draw=none},
      boxplot/draw direction=y,
      enlarge x limits=0.12,
      grid=none,
      axis line style={black},
      tick style={black},
      xtick={1,2,3,4},
      xticklabels={A-optimal,Linearised,EIG-inst.,EIG-tilt},
      x tick label style={rotate=40, anchor=north east, font=\small},
      tick label style={font=\small},
      label style={font=\small},
  ]

  \nextgroupplot[
      title={Overall},
      ylabel={Relative error (\%)},
  ]
  \AddBoxFromRow{0}
  \AddBoxFromRow{6}
  \AddBoxFromRow{9}
  \AddBoxFromRow{12}

  \nextgroupplot[
      title={First component},
  ]
  \AddBoxFromRow{1}
  \AddBoxFromRow{7}
  \AddBoxFromRow{10}
  \AddBoxFromRow{13}

  \nextgroupplot[
      title={Second component},
  ]
  \AddBoxFromRow{2}
  \AddBoxFromRow{8}
  \AddBoxFromRow{11}
  \AddBoxFromRow{14}

  \nextgroupplot[
      title={Overall (detailed)},
      xtick={1,2},
      xticklabels={EIG-inst.,EIG-tilt},
      ymax=#4,
      enlarge x limits=0.35,
  ]
  \AddBoxFromRow{9}
  \AddBoxFromRow{12}

  \end{groupplot}
  \end{tikzpicture}
  \caption{#2}
  \label{#3}
  \end{figure}
}

%% file: sections/introduction.tex
Optimal experimental design (OED) is concerned with determining experimental strategies that maximize the information gained about unknown parameters under limited experimental resources. Since the seminal works of Kiefer, Lindley, and others~\cite{Kiefer1958,Lindley1956,Stone1959}, OED has developed into a broad field with applications across many areas~\cite{Martijn2005,Huan_Jagalur_Marzouk_2024}.

In the present work, we focus on experimental design for controlled dynamical systems. More precisely, we consider experiments governed by controlled ordinary differential equations, where the design variables include both the observation times and the control applied to the system. Such settings are common, for instance, in applications such as systems biology and process engineering~\cite{Kreutz2009,Franceschini2008}. 
In this setting, we restrict attention to non-adaptive designs, that is, designs that must be fully specified before the experiment starts.

The Fisher Information Matrix {\color{blue}(FIM)} and the design criterion derived from it provide a natural tool to address this problem. In particular, their time-additive structure makes it possible to cast the design problem as an optimal control problem of Bolza type. This, in turn, allows one to use tools from optimal control theory, such as Pontryagin’s principle for the analysis~\cite{Sager2013}, as well as efficient numerical methods~\cite{Korkel2004}. 

However, FIM-based design criteria remain intrinsically local, as they quantify information only around a nominal parameter value. They can therefore lead to inefficient designs in strongly nonlinear settings or under non-Gaussian or multimodal priors. Bayesian optimal design addresses these limitations by defining the design objective in terms of the posterior distribution. Among these criteria, the expected information gain (EIG) provides a natural information-theoretic objective~\cite{chalonerverdinelli1995,Ryan2016}.

Yet, despite its conceptual appeal, the EIG is expensive to evaluate and optimize, since its computation typically involves nested expectations over future observations and posterior distributions~\cite{Rainforth2018}. Methods such as variational approaches~\cite{Foster2019} can reduce part of this burden by improving EIG estimation for a given design, but they generally address the evaluation of the criterion rather than the outer optimization difficulties arising here. In our setting, an additional difficulty is structural: the time-additive form exploited by FIM-based criteria is generally lost. Moreover, because the control shapes the entire system trajectory, early decisions affect both the subsequent state evolution and the information content of later measurements, so the design problem does not, in general, admit the kind of sequential or greedy decomposition that is possible in other settings~\cite{Krause2008, maio2025}.

Despite these difficulties, several works have addressed Bayesian design for nonlinear dynamical systems. Busetto et al. provided an early information-theoretic Bayesian design approach for model selection in nonlinear dynamical systems~\cite{Busetto2009}. More recently, Overstall et al. showed that fully Bayesian design can be carried out for ODE models by combining Monte Carlo expected-utility approximation with probabilistic ODE solvers and optimization over finite-dimensional sampling designs~\cite{Overstall2019}. Huan and Marzouk, and later Paulson et al., further improved the tractability of EIG-based design through polynomial-chaos surrogates, using global approximations in the former case and design-dependent local ones in the latter~\cite{huanmarzouk2013,Paulson2019}. 
Recent work has also explored alternative ways of using local score or
information structure in experimental design. For example, Qi and Baker combine
Fisher-based local sensitivity measures with Sobol'-based global sensitivity
measures to study how observation timing and correlated noise affect
parameter-estimation uncertainty in dynamic OED~\cite{qi2025observationnoise}.
Other approaches use functional inequalities, such as logarithmic Sobolev
bounds, to obtain tractable lower bounds on mutual information for greedy
selection from a finite candidate set~\cite{li2024logsobolev}. The present
work is different in scope and mechanism: we use Fisher information neither
as the final scalar design criterion nor as a mutual-information bound, but
as a deterministic tilting mechanism inside an EIG-motivated surrogate
designed for adjoint-compatible dynamic optimization.

In the present work, we take a different route. Rather than using surrogate models to accelerate the evaluation of an EIG-based objective, we introduce surrogate design criteria that recover temporal additivity and are therefore compatible with adjoint-based optimization in large-scale non-adaptive control problems. Our goal is thus not to speed up the computation of the EIG itself, but to obtain an optimization problem with exploitable structure while remaining close to the original Bayesian objective. To this end, we derive two criteria from a time decomposition of the EIG via the chain rule for mutual information, combined with posterior approximations chosen to recover the adjoint-compatible structure of the FIM-based optimal design problem.

The first surrogate criterion considered is myopic and replaces each intermediate posterior distribution by the prior, so that each new observation is quantified as if it were the first one. The second relies on a tilted posterior approximation driven by the Fisher information matrix, in order to account for directions already explored while preserving a tractable dynamical structure. Importantly, the Fisher information matrix is used here only as an auxiliary state variable, and not as the final design criterion.

We provide a theoretical analysis of both criteria, identifying regimes in which the myopic surrogate is effective as well as situations in which it fails by counting already acquired information as new. We also show that the tilting surrogate is exact in the linear-Gaussian setting. Finally, numerical experiments on simulated examples illustrate the benefits of the proposed approach over standard FIM-based criteria, especially under non-Gaussian or multimodal priors, and show how the tilting mechanism improves on the myopic limitations.

In order to obtain a computationally tractable formulation, we discretize the
prior and observation-noise expectations using fixed weighted-node integration
rules. This yields a simple, fully differentiable optimal-control problem that
can be solved by standard nonlinear programming solvers. The main limitation
of this choice is dimensional: standard tensor-product rules quickly become
impractical as the number of uncertain parameters or active observation-noise
variables increases. Quasi-Monte Carlo sequences, which are used in our
numerical experiments, can improve practical scaling when the
integrands have low effective dimension~\cite{dick2013high}, and sparse grids
offer another possible alternative~\cite{bungartz2004sparse}. For 
genuinely high-dimensional parameter spaces, however, the proposed surrogates
would need to be combined with dimension-reduction, variational, or
transport-based approximations, which is outside the scope of this work.
Nevertheless, many OED problems for dynamical systems, such as pharmacokinetic sampling design ~\cite{mentre1997optimal}, chemical kinetics~\cite{korkel2004numerical}, and experimental epidemic design~\cite{pagendam2013optimal}, involve a relatively small number of unknown parameters and therefore fall within the regime addressed here.

The remainder of the paper is organized as follows. Section~\ref{sec:problem_formulation} introduces the controlled experimental design problem and the Bayesian design criteria considered in this work. Section~\ref{sec:surrogates} presents the proposed surrogate objectives. Section~\ref{sec:theory} is devoted to their theoretical analysis. Section~\ref{sec:optim_problem_formulation} details the optimization problems associated with the surrogates. Section~\ref{sec:numerical_results} presents numerical experiments.

%% file: sections/general_settings.tex
\subsection{Model and observation setting}

\paragraph{Dynamical model}

Let $x(t)\in\mathbb{R}^{n_x}$ denote the state of a dynamical system evolving over the time interval $[0,T]$.
The system dynamics are governed by the controlled ordinary differential equation
\begin{equation}
    \dot x(t) = f\bigl(x(t),u(t),\theta,t\bigr), \qquad t\in[0,T],
\end{equation}
where $u(t)\in\mathbb{R}^{n_u}$ is a control input chosen by the experimenter and
$\theta\in\mathbb{R}^{n_\theta}$ is an unknown but fixed parameter vector. We assume that $f$ is sufficiently regular so that, for any measurable control $u(\cdot)$ taking values in a compact set $U$ and any $\theta$, the system admits a unique solution for the initial condition $x(0)=x_0$, which we denote by $x(t;u,\theta)$.

\paragraph{Sampling design and observation model}

We begin by considering that observations can only be collected on a finite grid of candidate sampling times
\(
    \mathcal{T}_M := \{t_1,\dots,t_M\}\subset[0,T],
\)
reflecting practical limitations on data acquisition.

At each time $t_i$, the experimenter may activate any of the $n_{\mathrm{exp}}$ sensors (or observation channels). When sensor $d\in\{1,\dots,n_{\mathrm{exp}}\}$ is active at time $t_i$, the observation follows the additive-noise model
\begin{equation}\label{eq:obs-model}
  y_{i,d}=h_d\!\left(x(t_i;u,\theta)\right)+\varepsilon_{i,d},
\end{equation}
where $h_d:\mathbb{R}^{n_x}\to\mathbb{R}^{n_y}$ is the $d$-th observation function and $(\varepsilon_{i,d})$ are independent random variables with density $p_{\varepsilon_{i,d}}$. The noise is assumed independent of the state and parameter, so that the observations $y_{i,d}$ are conditionally independent across sensors and time given $\theta$. In the remainder of the paper, we denote the full set of potential observations by $y = (y_{i,d})_{i,d}$ and the subset collected at time $t_i$ by $y_i = (y_{i,d})_{d}$.

The experimental design consists of:
\begin{itemize}
    \item a control function $u(\cdot)$ driving the system dynamics;
    \item sampling weights $w_d:\mathcal{T}_M\to\{0,1\}$, where $w_d(t_i)$ indicates whether sensor $d$ is activated at time $t_i$. We use the shorthand $w_{i,d} := w_d(t_i)$.
\end{itemize}
We denote by
\(
w = (w_{i,d})_{1\le i \le M,\; 1\le d \le n_{\mathrm{exp}}}
\) and \(
w_i = (w_{i,d})_{1\le d \le n_{\mathrm{exp}}}
\)
the full set of sampling weights and the vector of weights at time $t_i$, respectively. Throughout this work, we restrict attention to non-adaptive designs: the weights and control cannot depend on the realization of the observations. We also assume that the total number of measurements is limited by a budget constraint
\[
    \sum_{i=1}^M \sum_{d=1}^{n_{\text{exp}}} w_{i,d} \le K.
\]
To account for inactive sensors, we extend the observation space to 
\(
y_i \in (\mathbb{R}^{n_y} \cup \{c\})^{n_{\mathrm{exp}}},
\)
where $c \notin \mathbb{R}^{n_y}$ is a cemetery value used to encode inactive
sensors. 
Under this convention, inactive sensors produce the deterministic value $c$,
while active sensors take values in $\mathbb{R}^{n_y}$. We denote by
\(
\eta = (\eta_1,\dots,\eta_{n_{\mathrm{exp}}}) \in \{0,1\}^{n_{\mathrm{exp}}}
\)
the sensor configuration, where $\eta_d = 1$ indicates that sensor $d$
is active and $\eta_d = 0$ that it is inactive. The configuration can
then be recovered from the observation vector as
\begin{equation} \label{eq:config_y}
\eta_d(y_i) := \mathbf{1}_{\{y_{i,d} \neq c\}} .
\end{equation}
Given a sampling policy $w_i$, we define a configuration weight $\pi_\eta(w_i)$ measuring the consistency between the configuration $\eta$ and the design as
\begin{equation} \label{eq:configuration_weight}
\pi_\eta(w_i)
:=
\prod_{d=1}^{n_{\mathrm{exp}}}
w_{i,d}^{\eta_d}
(1-w_{i,d})^{1-\eta_d}.
\end{equation}
By conditional independence of the sensors given the parameter,
the predictive density of the observations at time $t_i$ is defined as
\begin{equation} \label{eq:likelihood_def}
p(y_i \mid \theta, w, u)
:=
\pi_{\eta(y_i)}(w_i)
\prod_{d:\eta_d(y_i)=1}
p_d(y_{i,d} \mid \theta, u).
\end{equation}
where
\(
p_d(y_{i,d}\mid\theta,u)
=
p_{\varepsilon_{i,d}}
\!\left(y_{i,d}-h_d(x(t_i;u,\theta))\right).
\)
We assume throughout that all relevant densities exist and that all necessary
integrability conditions hold, so that all information-theoretic quantities
considered below are well defined and finite. Note that since the design variables are deterministic, conditioning on $w$ and $u$ is purely parametric.
 
\subsection{Design objectives}

\paragraph{Fisher-based objectives}

A classical approach to optimal experimental design for parameter estimation relies on the Fisher Information Matrix (FIM) \cite{Fisher1922,Atkinson2007}. The FIM can be expressed as the covariance of the score function, or equivalently as the expected negative Hessian of the log-likelihood under standard regularity conditions. It quantifies the local sensitivity of the likelihood with respect to the unknown parameters through a quadratic approximation of the log-likelihood around a nominal parameter value. In this framework, criteria such as D-, A-, or E-optimality are commonly used to reduce the matrix information measure to a scalar \cite{pukelsheim2006optimal}.

Under standard assumptions, the FIM for a design $(w,u)$ is additive in time, with instantaneous information contributions from each measurement determined by state sensitivities and observation noise. In dynamical systems, this property allows Fisher-based criteria to be formulated as optimal control problems by augmenting the state with sensitivity equations, the accumulated information matrix, and variables enforcing measurement budgets \cite{PronzatoPazman2013,WalterPronzato1997}. The objective then takes the form of a time-integrated functional, for which gradients can be computed efficiently using adjoint methods. This structure makes Fisher-based design particularly attractive for large-scale controlled dynamical systems. Following \cite{Sager2013}, this leads to the optimal control problem
\begin{equation} 
\begin{aligned} 
\min_{u(\cdot), w(\cdot)} \quad & \varphi\bigl(F(T)\bigr) \\
\text{subject to} \quad
& \dot x(t) = f\bigl(x(t),u(t),\theta,t\bigr), \\
& \dot{G}(t) = \frac{\partial f}{\partial x}\bigl(x(t), u(t), \theta, t\bigr) G(t)
               + \frac{\partial f}{\partial \theta}\bigl(x(t), u(t), \theta, t\bigr), \\
& \dot{F}(t) = \sum_{d=1}^{n_{\mathrm{exp}}}
w_d(t)\, \left[\frac{\partial h_d}{\partial x}\bigl(x(t)\bigr) G(t)\right]^\top
R_d ^{-1} \left[\frac{\partial h_d}{\partial x}\bigl(x(t)\bigr) G(t)\right], \\
& \dot{z}(t) = \sum_{d=1}^{n_{\mathrm{exp}}}w_d(t), 
x(0) = x_0, \quad G(0) = 0, \quad F(0) = 0, \\ 
& z(0) = 0, 
u(t) \in U, \quad w_d(t) \in [0,1], \quad
0 \le K - z(T), 
\end{aligned} \label{eq:fim_optimal_control}
\end{equation}
where $\theta$ is fixed at a nominal value, typically  the prior mean or mode, $G$ denotes the state sensitivity with respect to the parameters, $F$ the accumulated Fisher information matrix, $\varphi$ a scalar optimality criterion applied to the FIM, $R_d$ the noise covariance associated with sensor $d$, and $z$ an auxiliary variable enforcing a measurement budget constraint with upper bound $K$.

This formulation is obtained by relaxing the sampling decision functions from discrete mappings $w_d : \mathcal{T}_M \to \{0,1\}$ to measurable functions $w_d : [0,T] \to [0,1]$. Such continuous relaxations are well established in optimal experimental design, as they convexify the design space and enable efficient gradient-based optimization methods \cite{Kiefer2018}.

Despite their computational efficiency and widespread use, Fisher-based criteria are intrinsically local, as they rely on a nominal parameter value and on a quadratic approximation of the log-likelihood that implicitly assumes near-Gaussian posterior behavior. Consequently, they can be unreliable in strongly nonlinear settings, under large prior uncertainty, or when multiple parameter modes are plausible \cite{PronzatoPazman2013,chalonerverdinelli1995}. Moreover, many Fisher-based criteria assume that the information matrix is nonsingular, an assumption that may fail when parameters are poorly identifiable or the experiment provides limited information. As a result, Fisher-based designs may fail to maximize the information gained from an experiment, which has led to the development of information-theoretic criteria that directly quantify expected uncertainty reduction, such as the expected information gain.

\paragraph{Expected information gain}

To overcome the limitations of Fisher information--based criteria, one can instead adopt
design objectives that account for the full experimental setup by explicitly considering the
posterior distribution \cite{RaiffaSchlaifer1961,Lindley1972}. By Bayes' rule, the posterior combines prior knowledge with the
observation likelihood and thus captures the overall informativeness of the experiment.
Criteria based on this principle are commonly referred to as the framework of Bayesian
optimal experimental design \cite{chalonerverdinelli1995}.

Among Bayesian optimal design criteria, the expected information gain (EIG) is one of the most widely
used objectives \cite{Ryan2003}. From an information-theoretic perspective, it quantifies the expected
reduction in uncertainty about the unknown parameter induced by the experiment. Formally,
the EIG is defined as the mutual information between the parameter $\theta$ and the
observation data $y$, conditioned on the design variables
$(w,u)$:
\[
J_{\mathrm{EIG}}(w,u)
= I(\theta; y \mid w,u)
= \mathbb{E}
\left[
\log \frac{p(y \mid \theta,w,u)}{p(y \mid w,u)}
\right],
\]
where the expectation is taken with respect to the joint distribution
$p_0(\theta)\,p(y \mid \theta,w,u)$.

Equivalently, the EIG can be interpreted as the expected Kullback--Leibler divergence between
the posterior distribution and the prior.

In contrast to Fisher-based criteria, the expected information gain remains valid in nonlinear settings and under non-Gaussian posterior distributions. However, this broader applicability comes at a significant computational cost: evaluating the
EIG requires high-dimensional expectations over both parameters and observations and
typically involves repeated posterior updates \cite{Ryan2016,huanmarzouk2013}. Moreover, the resulting objective does not, in general, admit a tractable time-additive structure or a formulation compatible with
adjoint-based optimal control methods.
This loss
of structure severely limits its direct use in controlled dynamical settings and motivates the
development of tractable approximations, which we address in the next section.

%% file: sections/eig_approximations.tex
\subsection{Chain-rule decomposition of the expected information gain}

The main obstacle to using the expected information gain in adjoint-based optimal control
frameworks is the lack of a tractable time-additive structure. In contrast to Fisher-based criteria, the
EIG couples all observations through the posterior distribution and therefore cannot be
directly expressed as an integral of instantaneous contributions. To recover a temporally
resolved representation, we rely on the chain rule for mutual information \cite{coverthomas1991}.

Collecting the observations up to time $t_i$ in the vector
\(
y_{1:i} := (y_1,\dots,y_i),
\)
the expected information gain can be decomposed exactly with respect to time as
\begin{equation}
  J_\mathrm{EIG}(w,u)
  =
  \sum_{i=1}^{M}
  I\bigl(\theta; y_i \,\big|\, y_{1:i-1},w,u\bigr).
  \label{eq:MI-chain-rule}
\end{equation}
Each term in this sum represents
the incremental information brought by the $i$-th observation, conditioned on all previous
measurements, and admits the expression
\begin{equation}
  I\bigl(\theta; y_i \,\big|\, y_{1:i-1},w,u\bigr)
  =
  \E_{\theta, y_{1:i}}\biggl[
    \log \frac{p\bigl(y_i \mid \theta, y_{1:i-1},w,u\bigr)}
              {p\bigl(y_i \mid y_{1:i-1},w,u\bigr)}
  \biggr],
  \label{eq:MI-increment-conditional}
\end{equation}
where the expectation is taken with respect to the joint distribution
$p_0(\theta)\,p(y_{1:i} \mid \theta,w,u)$.

If no sensor is activated at time $t_i$, the observation $y_i$ is understood as a degenerate random variable carrying no information, the corresponding mutual information contribution is zero.

Equations~\eqref{eq:MI-chain-rule} and \eqref{eq:MI-increment-conditional} provide an exact
discrete-time representation of the EIG. However, except for very restrictive settings,
this expression is intractable, as evaluating it for a given design would require simulating full observation sequences and repeatedly updating the posterior distribution, operations that are computationally prohibitive in controlled dynamical systems.

Using conditional independence of the observations, \eqref{eq:MI-increment-conditional}
can be rewritten as
\begin{equation}
  I\!\left(\theta; y_i \,\middle|\, y_{1:i-1},w,u\right)
  =
  \E_{\theta, y_{1:i}}\!\left[
    \log
    \frac{
      p\!\left(y_i \mid \theta,w,u\right)
    }{
      \int p\!\left(y_i \mid \theta',w,u\right)\,
           p(\theta' \mid y_{1:i-1},w,u)\,\mathrm{d}\theta'
    }
  \right].
  \label{eq:MI-increment-Gaussian-rewrite}
\end{equation}
This formulation makes explicit that the only dependence on past observations enters through
the posterior distribution $p(\theta \mid y_{1:i-1},w,u)$. Consequently, any attempt to
construct a time-additive and adjoint-compatible approximation of the expected information
gain must necessarily rely on suitable approximations of this posterior, which is the focus
of the next subsection.

\subsection{Posterior approximations and adjoint-compatible surrogates}

We introduce two adjoint-compatible surrogate objectives obtained by approximating the posterior distribution in a way that removes the dependence on the realized observation history. Depending on the approximation, the posterior is either fixed a priori or governed by a system of differential equations independent of the observed data. This leads to time-additive criteria that admit a natural continuous-time formulation and are therefore suitable for adjoint-based optimization.

\paragraph{Instantaneous surrogate}
\label{subsec:approx_instant_gain}

A tractable and adjoint-compatible objective can be obtained by adopting a myopic approximation of the posterior distribution, which yields a time-additive structure that can be evaluated without sequential posterior updates. Myopic or one-step-ahead strategies are classical in Bayesian optimal experimental design, where each observation is selected based on its immediate information gain \cite{Cavagnaro2010,Drovandi2014}. In static sensor placement problems, this approach is often supported by submodularity, which justifies greedy sequential selection with theoretical guarantees \cite{Krause2008}. In the present controlled dynamical setting, however, the control input shapes the system trajectory and thereby affects the informativeness of future measurements. As a result, marginal information gains depend on the overall control policy, and the design cannot be constructed sequentially in a consistent greedy fashion. We therefore adopt a global approximation in which the posterior is assumed independent of past observations, and replace the conditional increments in the chain-rule decomposition by unconditional, or \emph{instantaneous}, mutual information:
\begin{equation} \label{eq:I_inst_def}
  I\!\left(\theta; y_i \,\middle|\, y_{1:i-1}, w, u\right)
  \;\approx\;
  I_{\mathrm{inst}}\!\left(\theta; y_i \,\middle|\, w, u\right) :=
  \E\!\left[
    \log p\!\left(y_i \mid \theta, w, u\right)
    - \log p\!\left(y_i \mid w, u\right)
  \right],
\end{equation}
where the expectation is taken with respect to the joint distribution
\( p_0(\theta)\, p(y_i \mid \theta, w, u) \). Here, ``instantaneous'' refers to a single observation stage in
the chain-rule decomposition: each observation is evaluated as if it were the
first one.

Under this approximation, the chain-rule decomposition reduces to a sum of unconditional information gains, yielding the additive objective

\[
  J_{\mathrm{inst}}(w, u)
  =
  \sum_{i=1}^M
  I_{\mathrm{inst}}\!\left(\theta; y_i \mid w, u\right).
\]
This approximation amounts to replacing integration with respect to the posterior
by integration with respect to the prior. While this assumption discards the
progressive concentration of the posterior as information accumulates, it yields a
simple and computationally efficient surrogate objective that preserves the
structural advantages of Fisher-based design criteria. The quality of this
approximation and its relationship to the exact expected information gain are
analyzed in Subsection~\ref{subsection:theory_instant}.

\paragraph{Gaussian tilting surrogate}

To retain adjoint compatibility while accounting for posterior contraction and mitigating performance degradation and clustering effects, we introduce a refined surrogate based on a Gaussian tilting approximation. The construction is motivated by the linear--Gaussian setting, in which successive observations update the posterior through additive contributions to the precision matrix. Rather than explicitly propagating a data-dependent posterior distribution, we approximate the cumulative effect of past observations by reweighting the prior with a design-driven quadratic information factor.

Let \(p_0(\theta)\) denote the prior distribution of \(\theta\). 
For each potential observation time \(t_i\), we define the instantaneous Fisher information contribution
\[
F_i^{\Delta} := \sum_{d=1}^{n_{\mathrm{exp}}} 
w_{i,d}\, H_{i,d}^\top R_{i,d}^{-1} H_{i,d},
\]
where \(H_{i,d}\) denotes the Jacobian of the \(d\)-th observation map with respect to \(\theta\), evaluated at a fixed reference point \(\theta_\mathrm{ref} \in \mathbb{R}^{n_\theta}\)
\[
H_{i,d}
=
\left.
\frac{\partial}{\partial \theta}  h_d(x(t_i;\theta,u))
\right|_{\theta=\theta_\mathrm{ref}}
=
\frac{\partial h_d}{\partial x}\bigl(x(t_i;\theta_\mathrm{ref},u)\bigr)\,
G(t_i),
\]
and 
\[
G(t_i) = \left.\frac{\partial}{\partial \theta} x(t_i;\theta,u)\right|_{\theta=\theta_\mathrm{ref}}
\]
is the state sensitivity matrix introduced in \eqref{eq:fim_optimal_control}. Several choices are possible for the reference point, such as the prior mean or mode. Unless otherwise specified, \(\theta_\mathrm{ref}\) is taken to be the prior mean throughout the paper.

Given a design \((w,u)\), we introduce the accumulated Fisher information matrix, 
\begin{equation} \label{eq:accumulator_definition}
F_0 = 0,
\qquad
F_{i+1} = F_{i} + F_{i+1}^\Delta,
\end{equation}
representing the Fisher information accumulated up to stage $i$.

We define the Gaussian information factor
\[
\phi_i(\theta)
\;=
\exp\!\left(
-\tfrac12(\theta-\theta_\mathrm{ref})^\top F_i(\theta-\theta_\mathrm{ref})
\right).
\]
This factor mimics the cumulative contraction of the posterior induced by past observations. 
In particular, in the linear--Gaussian setting successive observations
lead to additive updates of the precision matrix. If the prior is Gaussian \(p_0 = \mathcal{N}(\theta_\mathrm{ref},\Sigma_0)\), tilting by \(\phi_i\) yields a Gaussian distribution with covariance \(\Sigma_i\) satisfying
\(
\Sigma_i^{-1} = \Sigma_0^{-1} + F_i .
\)

We then define the surrogate posterior at stage \(i\) by tilting and renormalizing the prior,
\begin{equation}
q_i(\theta)
:=
\frac{
  \phi_i(\theta)\,p_0(\theta)
}{
  \displaystyle
  \int \phi_i(\theta')\,p_0(\theta')\,\mathrm d\theta'
}. \label{eq:tilt_mechanism}
\end{equation}
By construction, \(q_i\) depends on the design variables only through the accumulated information matrix \(F_i\), which evolves additively along the experiment. This structure can be embedded as an auxiliary state variable, thereby preserving adjoint compatibility.

Using this surrogate within the chain-rule decomposition, we define the Gaussian tilting surrogate of the expected information gain as
\begin{equation}
J_{\mathrm{tilt}}(w,u)
=
\sum_{i=1}^M
\mathbb{E}\left[
  \log
  \frac{
    p\!\left(y_i \mid \theta, w, u\right)
  }{
    \displaystyle
    \int
    p\!\left(y_i \mid \theta', w, u\right)\,
    q_{i-1}(\theta')\,
    \mathrm d\theta'
  }
\right],
\end{equation}
where the expectation is taken with respect to the joint distribution
$q_{i-1}(\theta)\,p(y_i \mid \theta,w,u)$.

To ensure coherence of the proposed surrogate, consistency in the linear--Gaussian case is established in Subsection~\ref{subsec:tilt-theory}.
\begin{remark}[Distinction from the Laplace approximation]
The proposed construction should not be confused with a Laplace approximation of the posterior, often used in EIG computations \cite{Long2013,Lewi2009}. A Laplace approach builds a Gaussian approximation centered at a data-dependent mode, leading to updates of both the mean and the covariance. In contrast, the present tilting strategy introduces a deterministic quadratic reweighting of the prior that mimics a precision update while leaving the mean unchanged. This choice is motivated by the linear-Gaussian case, where the expected information gain depends only on the contraction of the posterior covariance and is independent of the posterior mean. It is this distinction that allows adjoint compatibility.
\end{remark}

%% file: sections/theoretical_results.tex
This section analyzes the approximation properties of the proposed surrogates of the expected information gain.
For the instantaneous surrogate, we establish a precise error identity showing that it systematically overestimates the true expected information gain, and we derive explicit upper and lower bounds that quantify the role of temporal redundancy under budget constraints.
For the Gaussian tilting surrogate, we prove exact consistency in the linear–Gaussian setting and establish stability with respect to prior approximations converging in the 2–Wasserstein sense.
Together, these results clarify the regimes in which each surrogate provides a reliable approximation of the expected information gain and supply a theoretical justification for their use in controlled experimental design.

\subsection{Redundancy and bounds for the instantaneous surrogate} \label{subsection:theory_instant}

We begin by characterizing the error induced by the instantaneous surrogate
and show that it systematically overestimates the expected information gain.

\begin{lemma}[Redundancy of the instantaneous surrogate]
For each $i \in \{1,\dots,M\}$, the error induced by the instantaneous surrogate satisfies \label{lmm:redundancy}
\begin{equation}
  I_{\mathrm{inst}}\!\left(\theta; y_i \,\middle|\, w,u\right)
  -
  I\!\left(\theta; y_i \,\middle|\, y_{1:i-1},w,u\right)
  =
  I\!\left(y_i; y_{1:i-1} \,\middle|\, w,u\right)
  \;\ge\; 0.
\end{equation}
As a consequence, the instantaneous surrogate provides an upper bound on each
incremental information gain, and therefore on the total expected information gain.
\end{lemma}
\begin{proof}
Expanding $I\!\left(y_i;\theta, y_{1:i-1}\mid w,u\right)$ with the chain rule in two different orders yields 
\begin{align}
I\!\left(y_i;\theta, y_{1:i-1}\mid w,u\right)
&= I\!\left(y_i;\theta\mid w,u\right) + I\!\left(y_i; y_{1:i-1}\mid \theta,w,u\right), \label{eq:order1}\\
I\!\left(y_i;\theta, y_{1:i-1}\mid w,u\right)
&= I\!\left(y_i; y_{1:i-1}\mid w,u\right) + I\!\left(y_i;\theta\mid y_{1:i-1},w,u\right). \label{eq:order2}
\end{align}
Subtracting \eqref{eq:order2} from \eqref{eq:order1} and rearranging gives
\[
I\!\left(y_i;\theta\mid w,u\right) - I\!\left(y_i;\theta\mid y_{1:i-1},w,u\right)
=
I\!\left(y_i; y_{1:i-1}\mid w,u\right) - I\!\left(y_i; y_{1:i-1}\mid \theta,w,u\right).
\]
Under the conditional independence across time of the observations given $\theta$,
the second term vanishes:
\(
I(y_i; y_{1:i-1} \mid \theta,w,u)=0.
\)
Since $I_{\mathrm{inst}}\!\left(\theta; y_i \,\middle|\, w,u\right) = I\!\left(y_i;\theta\mid w,u\right)$, the result follows.
\end{proof}
Since the observation noises are independent, the only source of correlation
between \(y_i\) and \(y_{1:i-1}\) is their shared dependence on the unknown parameter
\(\theta\).
Lemma \ref{lmm:redundancy} therefore shows that neglecting posterior updates amounts to treating successive
measurements as conditionally independent, which leads to systematic double
counting of information.
The resulting approximation error is thus large for highly redundant measurements
and remains small when observations are weakly correlated, for instance when
they are well separated in time or probe complementary system sensitivities.

We now show that, despite its myopic nature, the instantaneous surrogate
also provides a meaningful lower bound on the expected information gain under
natural design constraints. 
\begin{proposition}[Instantaneous surrogate bounds]
Consider an admissible design $(w,u)$ for which at least one measurement is taken, and define
\[
K_{\mathrm{time}}(w)
:=
\operatorname{card}
\Big\{
i\in\{1,\dots,M\} :
\sum_{d=1}^{n_{\mathrm{exp}}} w_{i,d} > 0
\Big\}.
\]
Then
\[
\frac{1}{K_{\mathrm{time}}(w)}\, J_{\mathrm{inst}}(w,u)
\;\le\;
J_{\mathrm{EIG}}(w,u)
\;\le\;
J_{\mathrm{inst}}(w,u).
\]
Moreover, the budget constraint $\sum\limits_{i=1}^{M} \sum\limits_{d=1}^{n_{\mathrm{exp}}} w_{i,d} \le K$ implies
$K_{\mathrm{time}}(w)\le K$, and therefore
\[
\frac{1}{K}\, J_{\mathrm{inst}}(w,u)
\;\le\;
J_{\mathrm{EIG}}(w,u)
\;\le\;
J_{\mathrm{inst}}(w,u).
\]
\end{proposition}
\begin{proof}
The upper bound follows directly from Lemma~\ref{lmm:redundancy}.

For the lower bound, let
\(
v_i := \mathbf{1}\Big\{\sum_{d=1}^{n_{\mathrm{exp}}} w_{i,d} > 0\Big\}
\) for $i\in \{1,\dots,M\}$. 
$K_{\mathrm{time}}$ is then equal to $\sum\limits_{i=1}^M v_i$.
By monotonicity of mutual information,
\[
I(\theta; y_{1:M}\mid w,u)\ge I(\theta; y_i\mid w,u)
\qquad \forall i.
\]
Multiplying by $v_i$ and summing yields
\[
K_{\mathrm{time}}(w)\, I(\theta; y_{1:M}\mid w,u)
\;\ge\;
\sum_{i=1}^M v_i\, I(\theta; y_i\mid w,u)
= J_{\mathrm{inst}}(w,u).
\]
Dividing by $K_{\mathrm{time}}(w)>0$ gives the result.
\end{proof}
These bounds clarify the regimes in which the instantaneous surrogate is
informative. Under a strict budget limiting the total number of observations,
$J_{\mathrm{inst}}$ provides a meaningful surrogate for the expected information
gain. Conversely, when the measurement budget is large and the selected measurements
are highly redundant, the instantaneous surrogate may become loose, in which
case $J_{\mathrm{inst}}$ primarily acts as an upper bound due to the neglect of
posterior updates.

\subsection{Consistency of the Gaussian tilting surrogate} \label{subsec:tilt-theory}

As a consistency check, we consider the classical linear--Gaussian setting.
In this case, we show that the proposed Gaussian tilting surrogate coincides
with the exact expected information gain when the prior is Gaussian.
We further establish that the surrogate is stable under prior
approximations converging in the $2$--Wasserstein ($W_2$) sense,
thereby ensuring robustness with respect to particle or quadrature-based
representations of the prior.

\begin{proposition}[Consistency in the Linear--Gaussian setting]
\label{prop:consistency_LG}
Consider a fixed design $(w, u)$ providing, for some observation matrix $H_{i, d}$ and  affine offset $b_{i,d}$, the linear--Gaussian observation models
\[
y_{i, d} = H_{i, d} \theta + b_{i, d} + \varepsilon_{i, d},
\qquad \varepsilon_{i, d} \sim \mathcal N(0,R_{i, d}),
\quad i=1,\dots,M, ~~\text{and}~~ d=1,\dots,n_{exp}
\]
with positive definite $R_{i, d}$. Let $J_{\mathrm{EIG}}(p_0)$ denote the exact expected information gain for this model under the parameter prior $p_0$, and let $J_{\mathrm{tilt}}(p_0)$ denote the Gaussian tilting objective constructed
from $p_0$.
\begin{enumerate}
\item
If the prior is Gaussian, $p_0=\mathcal N(m_0,\Sigma_0)$, then the Gaussian tilting
surrogate with $\theta_\mathrm{ref}=m_0$ is exact, and
\[
J_{\mathrm{tilt}}(p_0) = J_{\mathrm{EIG}}(p_0) = \tfrac{1}{2}\log\det(I + \Sigma_0 F_M),
\]
where \(F_M\) is the final accumulated Fisher information matrix defined
by~\eqref{eq:accumulator_definition}.
\item
Let $(p_k)_{k\ge1}$ be a sequence of priors with finite second moment such that
$p_k \to p_0:=\mathcal N(m_0,\Sigma_0)$ in $W_2$. Then, with \(\theta_{\mathrm{ref}}^{(k)}\) set to the mean of \(p_k\), we have,
for each fixed design, the pointwise convergence
\[
J_{\mathrm{tilt}}(p_k)
\xrightarrow[k\to\infty]{}
J_{\mathrm{EIG}}(p_0).
\]
\end{enumerate}
\end{proposition} 
\begin{proof}
We prove the two assertions sequentially.

\noindent\textit{1. Exactness in the Gaussian case.}
Assume that the prior is Gaussian, $p_0=\mathcal N(m_0,\Sigma_0)$.
In the linear--Gaussian model, the posterior distribution after $i$
observations remains Gaussian, with precision matrix
\[
\Sigma_i^{-1}
=
\Sigma_0^{-1}
+
\sum_{j=1}^{i}
\sum_{d=1}^{n_{exp}}
w_{j,d}\, H_{j,d}^\top R_{j,d}^{-1} H_{j,d}
=
\Sigma_0^{-1}+F_i,
\]
where $F_i$ is the accumulated Fisher information defined in \eqref{eq:accumulator_definition}.
In particular, the posterior covariance $\Sigma_i$ is deterministic and
independent of the realized data.

The exact expected information gain admits the closed--form expression
\begin{equation}\label{eq:eig_gauss}
J_{\mathrm{EIG}}(p_0)
=
\sum_{i=1}^M \frac12 \log \det\!\left(I + \Sigma_{i-1} F_i^\Delta \right),
\end{equation}
where $\Sigma_{i-1}$ denotes the posterior covariance after $i-1$ steps.

Consider now the Gaussian tilting surrogate. Since $p_0$ is centered at $m_0$ and
$\theta_{\mathrm{ref}}=m_0$, tilting by
\(
\phi_i(\theta)
=
\exp\!\left(-\tfrac12(\theta-m_0)^\top F_i(\theta-m_0)\right)
\)
yields a Gaussian distribution $q_i$ with precision matrix $\Sigma_0^{-1}+F_i$ and hence
covariance $\Sigma_i$. Therefore, although the deterministic surrogate $q_i$ does not track the
data-dependent mean of the true posterior, it reproduces its exact covariance.
Since in the linear--Gaussian setting the incremental information gain depends
only on the prior covariance, the contribution computed from $q_{i-1}$ coincides
with the exact contribution given by \eqref{eq:eig_gauss}. Summing over
$i=1,\dots,M$ yields
\[
J_{\mathrm{tilt}}(p_0)=J_{\mathrm{EIG}}(p_0).
\]

\medskip

\noindent\textit{2. Convergence under $W_2$.}
Let $(p_k)_{k\ge1}$ be a sequence of probability measures with finite second moments such that
\(
p_k \to p_0 := \mathcal N(m_0,\Sigma_0)
\) in the $2$--Wasserstein distance.
Denoting by $m_k=\mathbb E_{p_k}[\theta]$ the mean of $p_k$, we define the tilting factors
\[
\phi_i^{(k)}(\theta)
=
\exp\!\left(-\tfrac12(\theta-m_k)^\top F_i (\theta-m_k)\right) \quad \text{and} \quad \phi_i(\theta)
=
\exp\!\left(-\tfrac12(\theta-m_0)^\top F_i (\theta-m_0)\right),
\]
as well as the probability measures
\[
q_i^{(k)}(d\theta)
\propto
\phi_i^{(k)}(\theta)p_k(d\theta)\quad \text{and} \quad 
q_i^{(0)}(d\theta)
\propto
\phi_i(\theta)p_0(d\theta).
\]
From Lemma~\ref{lmm:weak_conv_qk}, we obtain that for any sequence of continuous functions $(f_k)$ converging locally uniformly to $f$ and satisfying a uniform quadratic bound, we have
\begin{equation}
\mathbb{E}_{q_i^{(k)}}[f_k(\theta)]
\;\longrightarrow\;
\mathbb{E}_{q_i^{(0)}}[f(\theta)]. \label{eq:convergence_qk}
\end{equation}
For each stage \(i\), let \(Y_i\) denote the observation block at step \(i\),
with conditional density
\[
\ell_i(y\mid\theta):=p(Y_i=y\mid\theta,w,u).
\]
To conclude, it remains to verify that the functions
\[
f_i(\theta)
:=
-\mathbb{E}_{Y_i\sim \ell_i(\cdot\mid\theta)}
[\log \ell_i(Y_i\mid\theta)]
\]
and
\[
g_{i,k}(\theta)
:=
-\mathbb{E}_{Y_i\sim \ell_i(\cdot\mid\theta)}
[\log m_{i,k}(Y_i)] \quad \text{where} \quad m_{i,k}(y)
=
\int \ell_i(y\mid\theta')\,q_{i-1}^{(k)}(d\theta')
\]

satisfy a uniform quadratic bound in \(\theta\), and that \(g_{i,k}\) converges locally uniformly to 
\[
g_i(\theta)
:=
-\mathbb{E}_{Y_i\sim \ell_i(\cdot\mid\theta)}
[\log m_i(Y_i)]
\qquad \text{where} \quad
m_i(y)
=
\int \ell_i(y\mid\theta')\,q_{i-1}^{(0)}(d\theta').
\]
In the linear--Gaussian observation model the log-likelihood
$\log \ell_i(y\mid\theta)$ grows at most quadratically in $(\theta,y)$, which directly yields the uniform quadratic bound on $f_i$. The desired properties on \(g_{i,k}\) follow from Lemma~\ref{lmm:behavior_denom_eig_gauss}.

We can thus apply \eqref{eq:convergence_qk} to the first term with the fixed function \(f_i\), and to the second term with \(g_{i,k}\to g_i\), to conclude that the \(i\)-th stage contribution converges. Summing over \(i=1,\dots,M\), we obtain
\[
\lim_{k\to\infty} J_{\mathrm{tilt}}(p_k)
=
J_{\mathrm{tilt}}(p_0).
\]
Combining this with Part~1 yields
\[
\lim_{k\to\infty} J_{\mathrm{tilt}}(p_k)
=
J_{\mathrm{EIG}}(p_0).
\]
\end{proof}

%% file: sections/in_practice.tex
In this section, we reformulate the proposed surrogate objectives as optimal control problems that are computationally tractable within an adjoint-based framework.

In the previous sections, observations were collected at a finite set of
candidate times, and the EIG decomposition was written as a sum over these
observations. In order to obtain a formulation compatible with adjoint-based
optimal-control solvers, we now introduce a continuous-time optimal-control
relaxation of this discrete design problem. This is in the same spirit as
continuous-time optimal-control formulations of sampling decisions in classical
Fisher-based OED~\cite{Sager2013}.

The surrogate objectives involve expectations over the parameter
distribution. To obtain a closed and differentiable dynamical system,
we approximate the prior distribution $p_{0}$ by a finite Dirac mixture
\begin{equation} \label{eq:dirac_approx_prior}
    p_{0} \approx \sum_{k=1}^{N} m_{k} \delta_{\theta_{k}} .
\end{equation}
This approximation replaces expectations with respect to the prior and the tilted surrogate distributions by finite weighted sums, thereby reducing the problem to the propagation of a finite set of deterministic state trajectories.
\[
x_k(t) = x(t; u, \theta_k)
\]
that can be propagated simultaneously. 

As a result, the optimization problem can be cast as a standard nonlinear optimal control problem with smooth dynamics, enabling efficient adjoint-based gradient computation.

In practice, such weighted-node approximations can be obtained via Monte Carlo sampling or by
deterministic integration rules such as Gauss--Hermite quadrature, sparse grids, or
quasi-Monte Carlo sequences \cite{Zhou2024,dick2013high,bungartz2004sparse}. Gauss--Hermite rules
are exact for polynomials up to degree \(2q-1\) in one dimension and can
converge rapidly for smooth low-dimensional integrands while quasi-Monte
Carlo methods can improve practical convergence for sufficiently regular
integrands with low effective dimension~\cite{dick2013high}. More generally,
the convergence properties of these integration rules depend on the
regularity of the corresponding integrands. In the numerical examples
considered here, the ODE solution and observation maps depend smoothly on
the parameters and, together with the Gaussian observation model, lead to
smooth integrands in the parameter and noise variables.

However, tensor-product rules suffer from the curse of dimensionality. For
example, a Gauss--Hermite rule with \(q=4\) nodes per dimension requires
\(N=16\) nodes for \(n_\theta=2\), but already \(N=1{,}024\) for
\(n_\theta=5\). The formulation is therefore intended for low
effective uncertainty dimension, rather than as a general high-dimensional
EIG estimator. The deterministic noise quadrature is similarly limited to low-dimensional
active observation blocks, but the time-additive structure of the surrogates
ensures that it is only needed over the noise at a single time step, not
over the full sequence.

This approximation provides a natural first optimization-compatible realization of the proposed criteria. Exploring more scalable prior representations while preserving this compatibility is a promising direction for future work, but falls outside the scope of the present paper.

Throughout this section, to simplify notation, we assume that the noise distribution of each sensor does not depend on time and omit the time index, writing $\varepsilon_d$ instead of $\varepsilon_{i,d}$.

\subsection{Probabilistic Relaxation of the Sampling Policy} \label{subsec:relaxation}

To enable the use of adjoint-based optimization algorithms, the original discrete design problem must first be relaxed, as is classically done in the OED literature~\cite{Huan_Jagalur_Marzouk_2024}. We first derive the relaxation in the discrete-time setting introduced in Section~\ref{sec:problem_formulation}. The binary sampling decisions are relaxed by allowing the weights to take values in the unit interval. The passage to continuous time will be introduced later when deriving the optimal control formulation.

We interpret $w_{d}(t_i)\in[0,1]$ as the instantaneous probability of activating sensor $d$ at time $t_i$. This induces a mixture over the sensor configurations $\eta\in\{0,1\}^{n_{\mathrm{exp}}}$, where we remind that $\eta_d=1$ indicates that sensor $d$ is active and $\eta_d=0$ that it is inactive.
Under the assumption that the activations are independent across sensors, the configuration weight \eqref{eq:configuration_weight} can then be interpreted as a probability of activating configuration $\eta$ :
\begin{equation}
\pi_\eta(w(t_i))
:=
\prod_{d=1}^{n_{\mathrm{exp}}}
w_d(t_i)^{\eta_d}
\bigl(1-w_d(t_i)\bigr)^{1-\eta_d}.
\end{equation}
With $\eta(y_i)$ defined in \eqref{eq:config_y}, the relaxed likelihood is then defined by
\begin{equation} \label{eq:likelihood_relax}
p\big(y(t_i) \mid \theta, w, u\big)
:=
\pi_{\eta(y(t_i))}(w(t_i))
\prod_{d:\eta_d(y(t_i))=1}
p_d\big(y_{d}(t_i) \mid \theta, u\big).
\end{equation}
This mixture likelihood representation then provides a continuous relaxation of the discrete design variables
that coincides with the original model of Section~\ref{sec:problem_formulation} whenever $w(t_i)\in\{0,1\}^{n_{\mathrm{exp}}}$. As shown in the following sections, this structure leads to tractable objective functions whose gradients can be computed efficiently using adjoint methods.

For large sensor arrays, the combinatorial growth of the configuration set 
$\{0,1\}^{n_{\mathrm{exp}}}$ may become computationally prohibitive. 
Alternative formulations based on continuous sensor-weighting or precision-scaling 
can avoid the explicit enumeration of sensor subsets and lead to scalable optimization 
problems~\cite{Haber2008,Alexanderian2014}. 
In the present work, however, we restrict attention to $n_{\mathrm{exp}}=2$ sensors. 
Excluding the null configuration, only three non-trivial configurations need to be 
evaluated, keeping the formulation computationally manageable.

\subsection{Entropy of the relaxed observation model}

Using the entropy decomposition of mutual information, the incremental information gain can be written as
\begin{equation}
I(\theta; y_i \mid y_{1:i-1}, w, u)
=
{\cal H}(y_i \mid y_{1:i-1}, w, u)
-
{\cal H}(y_i \mid \theta, w, u),
\end{equation}

where $\mathcal{H}(a)$ denotes the entropy of the random
variable~$a$ and $\mathcal{H}(a \mid b)$ denotes the conditional entropy
of~$a$ given~$b$.

The second term corresponds to the uncertainty induced by measurement noise. For a fixed sensor configuration
\(\bar\eta \in \{0,1\}^{n_{\mathrm{exp}}}\), only the active sensors contribute to the continuous observation uncertainty. Under the additive observation model \eqref{eq:obs-model}, with noise independent of the state and the parameter, the conditional distribution of each active measurement is a translation of the corresponding noise distribution. By translation invariance of differential entropy, this yields
\begin{equation}
{\cal H}(y_i \mid \eta=\bar\eta, \theta, w, u)
=
\sum_{d:\bar\eta_d = 1} {\cal H}(\varepsilon_d).
\end{equation}
In the relaxed formulation, the observation law is a mixture over sensor configurations with weights
\(\pi_{\bar\eta}(w(t_i))\). Since the active sensor configuration \(\eta\) can be recovered deterministically from the relaxed observation \(y_i\) through the cemetery value \(c\), one has
\(
{\cal H}(\eta \mid y_i,\theta,w,u)=0
\).
Hence, by the chain rule for conditional entropy,
\begin{equation}
{\cal H}(y_i \mid \theta,w,u)
=
{\cal H}(\eta \mid \theta,w,u)
+
{\cal H}(y_i \mid \eta,\theta,w,u).
\end{equation}
Expanding the second term with respect to the distribution of \(\eta\), we obtain
\begin{equation}
{\cal H}(y_i \mid \theta,w,u)
=
{\cal H}(\eta \mid \theta,w,u)
+
\sum_{\bar\eta\in\{0,1\}^{n_{\mathrm{exp}}}}
\pi_{\bar\eta}(w(t_i))\,{\cal H}(y_i \mid \eta=\bar\eta,\theta,w,u).
\end{equation}
Moreover, conditionally on \(w\), the random configuration \(\eta\) is independent of \(\theta\), so that
\begin{equation}
{\cal H}(\eta \mid \theta, w, u)
=
{\cal H}(\pi(w(t_i)))
:=
-
\sum_{\bar\eta \in \{0,1\}^{n_{\mathrm{exp}}}}
\pi_{\bar\eta}(w(t_i)) \log \pi_{\bar\eta}(w(t_i)).
\end{equation}
Combining the previous identities yields
\begin{equation} \label{eq:entropy_measure}
{\cal H}(y_i \mid \theta, w, u)
=
{\cal H}(\pi(w(t_i)))
+
\sum_{d=1}^{n_{\mathrm{exp}}} w_d(t_i)\, {\cal H}(\varepsilon_d).
\end{equation}
The conditional entropy of the relaxed observation therefore consists of two contributions: a term associated with the randomized sensor configuration, and a term induced by the sensor noise distributions. The latter depends only on the sensor noise distributions and on the sampling policy, and is independent of both the system dynamics and the prior approximation used in the surrogate objectives. Moreover, the configuration entropy appears symmetrically in the predictive entropy term and therefore cancels out in the mutual information. As a result, the remaining contribution admits a simple explicit expression whose extension to the continuous-time relaxation introduced below is immediate. In the following subsections, we therefore focus on deriving continuous-time optimal control formulations for the two surrogate objectives.

\subsection{Optimal Control Problem for the Instantaneous Surrogate}

To obtain a formulation compatible with adjoint-based optimal control methods, we introduce a continuous-time relaxation of the sampling policy.
Instead of optimizing the sensor activations at the discrete sampling times \(t_i\), we consider continuous activation functions
\[
w_d : [0,T] \to [0,1], \qquad d=1,\dots,n_{\mathrm{exp}}.
\]

On a grid with step size \(\Delta t\), the continuous notation is understood in
the Riemann-sum sense:
\[
\int_0^T \sum_d w_d(t)\,dt \approx \Delta t\sum_{i,d} w_{i,d}.
\]
Thus, \(K\) discrete measurements correspond to a continuous-time budget
\(\Delta t K\). Similarly, the continuous-time objective should be interpreted through the same
Riemann-sum convention when compared with its discrete counterpart.

As shown in the previous subsection, the discrete configuration entropy term \({\cal H}(\pi(w(t)))\) appears in both the conditional and predictive entropies, and therefore cancels out in the mutual information. The remaining predictive contribution can thus be evaluated configuration-wise.
For a fixed configuration \(\eta \in \{0,1\}^{n_{\mathrm{exp}}}\), let \(y_\eta(t)\) denote the continuous observation vector generated by the sensors active under configuration \(\eta\). We then define the corresponding conditional likelihood by
\begin{equation}
p_\eta\big(y_\eta(t)\mid\theta,u\big)
=
\prod_{d:\eta_d=1}
p_{\varepsilon_d}\!\left(y_d(t)-h_d\bigl(x(t;u,\theta)\bigr)\right),
\end{equation}
and the associated predictive density by
\(
p_\eta\big(y_\eta(t)\mid u\big)
=
\int_{\mathbb{R}^{n_\theta}}
p_\eta\big(y_\eta(t)\mid \theta',u\big)\,p_0(\theta')\,d\theta'.
\)

Using the definition of the instantaneous mutual information \eqref{eq:I_inst_def} with the likelihood relaxation \eqref{eq:likelihood_relax} and the entropy expression \eqref{eq:entropy_measure}, we define the continuous-time instantaneous information rate $I_{\mathrm{inst}}\!\left(\theta; y(t) \,\middle|\, w(t), u\right)$ as
\[
-
\sum_{\eta\in\{0,1\}^{n_{\mathrm{exp}}}}
\pi_\eta(w(t))
\,
\mathbb{E}_{\theta\sim p_0,\;y_\eta(t)\sim p_\eta(\cdot\mid\theta,u)}
\!\left[
\log p_\eta\big(y_\eta(t)\mid u\big)
\right]
-
\sum_{d=1}^{n_{\mathrm{exp}}} w_d(t)\, {\cal H}(\varepsilon_d).
\]
Since the instantaneous surrogate objective is additive over time, this yields the continuous-time objective
\begin{equation}
J_{\mathrm{inst}}(w,u)
=
\int_0^T
I_{\mathrm{inst}}\!\left(\theta; y(t)\,\middle|\, w(t),u\right)\,\mathrm{d}t.
\end{equation}
To evaluate the expectation over the observation noise, we use a quadrature rule (e.g., Gauss--Hermite) with nodes and weights $\{(s_q,\xi_q)\}_{q=1}^Q$, where $\xi_q\in\mathbb{R}^{n_{\mathrm{exp}}}$. 

The prior and noise expectations are discretized separately because the ODE
trajectories depend on~\(\theta\) but not on~\(\varepsilon\): for each current
design, the \(N\) parameter trajectories are propagated once and reused for all
\(Q\) noise nodes. This separation exploits the conditional structure of the
observation model and avoids coupling the expensive ODE propagation to the
noise quadrature.
Together with the Dirac approximation \eqref{eq:dirac_approx_prior} of the prior, this yields the following optimal problem

\begin{equation}
\begin{aligned}
\min_{u, w} \quad & -
\int_0^T \left[
\sum_{\substack{1\le k\le N\\ 1\le q\le Q\\ \eta\in\{0,1\}^{n_{\mathrm{exp}}}}}
m_k\, s_q\,
\pi_\eta\!\big(w(t)\big)
\,
\log\!\big(\mathcal L_{kq}^{\eta}(t)\big)
+ \sum_{d=1}^{n_{\mathrm{exp}}} w_{d}(t)\, {\cal H}(\varepsilon_d)\right]\,\mathrm{d}t , \\
\text{subject to} \quad
& \mathcal L_{kq}^\eta(t)
=
\sum_{\ell=1}^N
m_\ell
\prod\limits_{d:\eta_d = 1}
p_{\varepsilon_d}\big(
h_d(x_k(t))-
h_d(x_\ell(t))+
\xi_{q,d}
\big)\\
& \pi_\eta(w(t))
=
\prod_{d=1}^{n_{\mathrm{exp}}}
w_d(t)^{\eta_d}
\bigl(1-w_d(t)\bigr)^{1-\eta_d} \\
&\forall k \in \{1,\dots,N\}, \quad \dot{x}_k(t) = f\bigl(x_k(t), u(t), \theta_k, t\bigr), \\
& \dot{z}(t) = \sum_{d=1}^{n_{\mathrm{exp}}}w_d(t), \quad x_k(0) = x_0, \quad z(0) = 0, \\
& u(t) \in \mathcal{U}, \quad w_d(t) \in [0,1], \quad z(T) \le K.
\end{aligned} \label{eq:optimisation_prbl_inst}
\end{equation}

\begin{remark}[Bang--bang optimality of the sampling policy]
In the relaxed formulation, the sampling functions satisfy $w_d(t)\in[0,1]$.
For any fixed time $t$, the objective depends on $w(t)$ through multilinear
terms of the form $\prod_d w_d^{\eta_d}(1-w_d)^{1-\eta_d}$.
Since $\eta_d \in \{0,1\}$, such functions are affine in each component $w_d$ when the others are fixed.
Since the admissible set is a hypercube, extrema are attained at
its vertices \cite{LANEVE2010}. Consequently, whenever an optimal solution exists, there also
exists an optimal bang--bang sampling policy satisfying $w_d(t)\in\{0,1\}$.
\label{rem:bangbang_sol}
\end{remark}

\subsection{Optimal Control Problem for the Gaussian Tilting Surrogate}

We now introduce a continuous-time particle approximation of the Gaussian
tilting surrogate that is compatible with adjoint-based optimal control methods.

We consider the Fisher information matrix accumulated along a reference trajectory, denoted by $F(t)$. Following the Fisher-based optimal design problem
~\eqref{eq:fim_optimal_control}, we define
\[
\dot{F}(t)
=
\sum_{d=1}^{n_{\mathrm{exp}}}
w_d(t)\,
\left[\frac{\partial h_d}{\partial x}\bigl(x_\mathrm{ref}(t)\bigr) G(t)\right]^\top
R_d^{-1}
\left[\frac{\partial h_d}{\partial x}\bigl(x_\mathrm{ref}(t)\bigr) G(t)\right],
\]
where $F(0)=0$
and $G(t)$ denotes the state sensitivity with respect
to the parameters, satisfying
\[
\dot{G}(t)
=
\frac{\partial f}{\partial x}\bigl(x_\mathrm{ref}(t),u(t),\theta_\mathrm{ref},t\bigr)G(t)
+
\frac{\partial f}{\partial \theta}\bigl(x_\mathrm{ref}(t),u(t),\theta_\mathrm{ref},t\bigr),
\]
with $x_{\mathrm{ref}}(t):=x(t;u,\theta_{\mathrm{ref}})$ denoting the trajectory
associated with the fixed reference parameter $\theta_{\mathrm{ref}}$, which can
be taken to be the prior mean
\(
\theta_\mathrm{ref} := \sum_{k=1}^N m_k\,\theta_k .
\)
The associated Gaussian continuous weighting factor is then defined as
\[
\phi_t(\theta)
=
\exp\!\left(
-\tfrac12(\theta-\theta_\mathrm{ref})^\top F(t)(\theta-\theta_\mathrm{ref})
\right).
\]
Tilting and renormalizing the Dirac mixture prior yields the surrogate distribution
\begin{equation}
q_t^{(N)}(\theta)
=
\frac{\phi_t(\theta)\,p_0^{(N)}(\theta)}
{\int \phi_t(\theta')\,p_0^{(N)}(\theta')\,\mathrm d\theta'}
=
\sum_{k=1}^N \mu_k(t)\,\delta_{\theta_k}(\theta),
\label{eq:dirac-tilted-ct}
\end{equation}
with time-dependent weights
\(
\mu_k(t)
=
\frac{m_k\,\phi_t(\theta_k)}
{\sum_{j=1}\limits^N m_j\,\phi_t(\theta_j)}.
\)
Differentiating with respect to time, the weights satisfy the
replicator-type ordinary differential equation
\[
\dot \mu_k(t)
=
-\tfrac12\,\mu_k(t)
\left(
(\theta_k-\theta_\mathrm{ref})^\top \dot{F}(t)\,(\theta_k-\theta_\mathrm{ref})
-
\sum_{j=1}^N \mu_j(t)\,(\theta_j-\theta_\mathrm{ref})^\top \dot{F}(t)\,(\theta_j-\theta_\mathrm{ref})
\right),
\]
which preserves positivity and the simplex constraint
$\sum\limits_{k=1}^N \mu_k(t)=1$.

As in the instantaneous surrogate case, the discrete configuration entropy term \({\cal H}(\pi(w(t)))\) cancels out in the mutual information, and we approximate the expectation with respect to the measurement noise
using a quadrature rule. 
The adjoint-compatible optimal experimental design problem associated with the Gaussian tilting surrogate then reads
\begin{equation}
\begin{aligned}
\min_{u, w} \quad & -
\int_0^T \left[
\sum_{\substack{1\le k\le N\\ 1\le q\le Q\\ \eta\in\{0,1\}^{n_{\mathrm{exp}}}}}
\mu_k(t)\, s_q\,
\pi_\eta\!\big(w(t)\big)
\,
\log\!\big(\mathcal L_{kq}^{\eta}(t)\big)
+ \sum_{d=1}^{n_{\mathrm{exp}}} w_{d}(t)\, {\cal H}(\varepsilon_d)\right]\,\mathrm{d}t , \\
\text{s.t.} \quad
& \mathcal L_{kq}^\eta(t)
=
\sum_{\ell=1}^N
\mu_\ell(t)
\prod\limits_{d:\eta_d = 1}
p_{\varepsilon_d}\big(
h_d(x_k(t))-
h_d(x_\ell(t))+
\xi_{q,d}
\big)\\
& \pi_\eta(w(t))
=
\prod_{d=1}^{n_{\mathrm{exp}}}
w_d(t)^{\eta_d}
\bigl(1-w_d(t)\bigr)^{1-\eta_d} \\
&\forall k \in \{1,\dots,N\}, \quad \dot{x}_k(t) = f\bigl(x_k(t), u(t), \theta_k, t\bigr), \\
&\dot{x}_{\mathrm{ref}}(t) = f\bigl(x_{\mathrm{ref}}(t), u(t), \theta_{\mathrm{ref}}, t\bigr), \\
& \dot G(t)
= f_x\!\left(x_\mathrm{ref}(t),u(t),\theta_\mathrm{ref}, t\right) G(t)
+ f_\theta\!\left(x_\mathrm{ref}(t),u(t),\theta_\mathrm{ref}, t\right),\\
& \dot{F}(t) = \sum_{d=1}^{n_{\mathrm{exp}}}
w_d(t)\, \left[\frac{\partial h_d}{\partial x}\bigl(x_\mathrm{ref}(t)\bigr) G(t)\right]^\top
R_d^{-1}
\left[\frac{\partial h_d}{\partial x}\bigl(x_\mathrm{ref}(t)\bigr) G(t)\right], \\
& \forall k \in \{1,\dots,N\}, \quad \dot \mu_k(t)
= \\
&\quad \quad -\tfrac12\,\mu_k(t)
\left(
(\theta_k-\theta_\mathrm{ref})^\top \dot{F}(t)(\theta_k-\theta_\mathrm{ref})
-
\sum_{j=1}^N \mu_j(t)
(\theta_j-\theta_\mathrm{ref})^\top \dot{F}(t)(\theta_j-\theta_\mathrm{ref})
\right), \\
& \dot{z}(t) = \sum_{d=1}^{n_{\mathrm{exp}}}w_d(t), \quad
x_k(0)=x_0,\quad x_{\mathrm{ref}}(0)=x_0, \quad G(0)=0,\quad \mu_k(0)=m_k,\\ &z(0)=0, \quad
u(t) \in \mathcal{U}, \quad w_d(t) \in [0,1], \quad z(T) \le K.
\end{aligned} \label{eq:optimisation_prbl_tilt}
\end{equation}

Algorithm~\ref{algo:evaluation} summarizes the evaluation of the two surrogate
objectives. Gradients are obtained by reverse-mode automatic differentiation
of the same computational graph in \texttt{CasADi}.

\begin{algorithm}[htbp]
\caption{Evaluation of $J_{\mathrm{inst}}$ or $J_{\mathrm{tilt}}$}
\label{algo:evaluation}
\begin{algorithmic}[1]

\REQUIRE Prior nodes $\{(\theta_k,m_k)\}_{k=1}^N$, noise nodes
$\{(\varepsilon_q,\omega_q)\}_{q=1}^Q$, design $(u_i,w_i)_{i=1}^{N_t}$

\STATE Initialize particles $X_k=x_0$, objective $J=0$, and, for
$J_{\mathrm{tilt}}$, $X_{\rm ref}=x_0$, $G=0$, $F=0$

\FOR{$i=1,\ldots,N_t$}

\STATE Propagate all particles:
\[
X_k\leftarrow \Phi(X_k,\theta_k,u_i)
\]

\IF{$J_{\mathrm{tilt}}$}
\STATE Compute
\[
\mu_k \propto
m_k\exp\!\left(
-\frac12(\theta_k-\theta_{\rm ref})^\top
F(\theta_k-\theta_{\rm ref})
\right)
\]
\ELSE
\STATE Set $\mu_k=m_k$
\ENDIF

\STATE Normalize the weights $\mu_k$

\FOR{each non-empty sensor configuration $\eta$}

\STATE Approximate
\[
I_\eta
=
\sum_{q,k}\omega_q\mu_k
\left[
\ell_{kk}^{\eta,q}
-
\log\sum_{\ell=1}^N
\mu_\ell\exp(\ell_{k\ell}^{\eta,q})
\right],
\]
where, for Gaussian observation noise,
\[
\ell_{k\ell}^{\eta,q}
=
-\sum_{d:\eta_d=1}
\frac{
\bigl(h_d(X_k)-h_d(X_\ell)+\varepsilon_{q,d}\bigr)^2
}{
2\sigma_d^2
}.
\]

\ENDFOR

\STATE Accumulate
\[
J\leftarrow J+\sum_{\eta\neq0}\pi_\eta(w_i)I_\eta .
\]

\IF{$J_{\mathrm{tilt}}$}
\STATE Propagate the reference state/sensitivity system and update
the Fisher accumulator $F$ using $(u_i,w_i)$
\ENDIF

\ENDFOR

\RETURN $J$

\end{algorithmic}
\end{algorithm}

%% file: sections/numerical_results.tex
In this section, we assess the proposed design criteria on four numerical test cases built from two benchmark controlled dynamical systems.
The optimization problems associated with the instantaneous
surrogate~\eqref{eq:optimisation_prbl_inst}, the Gaussian tilting
surrogate~\eqref{eq:optimisation_prbl_tilt}, and the \(A\)-optimality
criterion~\eqref{eq:fim_optimal_control} were implemented in Python using
\texttt{CasADi}~\cite{Andersson2019} and solved with
\texttt{IPOPT}~\cite{Waechter2006}. We also include a local linear-Gaussian EIG baseline, defined by
\[
  J_{\mathrm{Linearised
}}(w,u)
  =
  \tfrac{1}{2} \log\det\!\left(I+\Sigma_0\, F(T)\right),
\]
where \(\Sigma_0\) is the prior covariance and \(F(T)\) is the accumulated
Fisher information along the reference trajectory. This criterion corresponds
to the EIG of the local linear--Gaussian approximation obtained by linearizing
the nonlinear model around the reference trajectory. It uses the same
augmented sensitivity/Fisher ODE system as the \(A\)-optimality criterion.
All reported runs terminated with IPOPT status ``Optimal Solution Found''.

We also tested a direct sample-average approximation of the full-sequence EIG (SAA-EIG).
In our implementation, the Bernoulli sensor relaxation led to unstable adjoint
evaluations, while a simpler power-likelihood relaxation converged but produced
less accurate designs in the examples considered. We therefore do not include
direct SAA-EIG optimization as a main baseline. This should however not be interpreted
as a general claim that direct EIG optimization is infeasible.

For each scenario, the resulting designs are evaluated by Monte Carlo simulation: 1000 independent parameter samples are drawn from the prior distribution and, for each sampled value, a maximum-likelihood estimate is computed under each design and compared with the true parameter value.

\subsection{Harmonic Oscillator}

We first consider a benchmark problem consisting of two second-order oscillators driven by a common control input over the time interval \([0,10]\). This example is intended to highlight two effects: first, the advantage of the proposed EIG-based surrogates over classical Fisher-based designs; second, the ability of the tilting surrogate to avoid over-concentrating the measurements on the most informative sensor.

The state variable is defined by the positions and velocities of the two oscillators,
\[
x(t)=\bigl(q_1(t),\,q_2(t),\,\dot q_1(t),\,\dot q_2(t)\bigr)^\top \in \mathbb{R}^4,
\]
and the dynamics are given by
\begin{equation}
\left\{
\begin{aligned}
\ddot q_1(t) + 0.5\,\dot q_1(t) + \theta_1 q_1(t) &= u(t),\\
\ddot q_2(t) + 0.5\,\dot q_2(t) + \theta_2 q_2(t) &= u(t),
\end{aligned}
\right.
\end{equation}
where \(u(t)\in[0,1]\) is a piecewise-constant control input, assumed to be constant on each element of a uniform partition of \([0,10]\) into \(12\) subintervals, and \(\theta = (\theta_1,\theta_2)^\top \in [5,10]^2\)
is the unknown parameter. The initial state is set to
\(x(0) = (1,\,1,\,0,\,0)^\top.\)

Two sensors are available, each measuring the position of one oscillator:
\begin{equation}
y_1(t) = q_1(t) + \varepsilon_{1}(t)
\quad\text{and}\quad
y_2(t) = q_2(t) + \varepsilon_{2}(t),
\end{equation}
where \(\varepsilon_{1}(t)\) and \(\varepsilon_{2}(t)\) are centered Gaussian random variables with standard deviations \(\sigma_1\) and \(\sigma_2\), respectively.

The prior distribution of \(\theta\) is uniform on \([5,10]^2\). The experimental design is restricted to \(8\) distinct observation times. The final design is obtained by discretizing the continuous optimal design \(w\), with the additional constraint that any two selected observation times must be separated by at least \(0.1\) time units.

The prior is discretized using a 32-point Sobol quasi-Monte Carlo rule, and
the scalar Gaussian measurement noise is discretized using a degree-4
Gauss--Hermite quadrature rule. The numbers of prior nodes \(N\) and noise quadrature nodes \(Q\) were
selected from a convergence study of the surrogate objectives evaluated
at the A-optimal designs. We increased \(N\) and \(Q\) until the surrogate
values reached a stable plateau, and used the corresponding discretizations
in the numerical experiments. Figure~\ref{fig:scaling} provides a further numerical
verification by showing that the surrogate values at the optimized designs
also stabilize around the selected values of \(N\) and \(Q\). For the Gaussian tilting surrogate and for the Fisher-based designs, the nominal parameter is chosen as the prior mean \((7.5,7.5)\).

\paragraph{Even observability}

In the first test case, the two components  are observed with the same noise level, \(\sigma_1=\sigma_2=0.1\), making the two observation channels equally informative.
The empirical distributions of the estimation errors over \(1000\) Monte Carlo runs are shown in Figure~\ref{fig:harmonic_similar_obs_errors}.

\ErrorBoxplotFigure
  {csv/harmonic_similar_obs_errors.csv}
  {Empirical distributions of the parameter-estimation errors for the harmonic oscillator test case with similar observability over 1000 Monte Carlo runs. The box spans the interquartile range, whiskers show the 10th and 90th percentiles, the orange line denotes the median and the triangle
marks the mean. Both EIG-based surrogates perform similarly in this balanced regime and achieve the best overall accuracy.}
  {fig:harmonic_similar_obs_errors}

As seen in Figure~\ref{fig:harmonic_similar_obs_errors}, both EIG surrogates perform similarly in this balanced regime. In particular, although the instantaneous surrogate still tends to favor the first sensor, it nevertheless distributes measurements across both channels sufficiently well to preserve a satisfactory balance in observability. In this scenario, all designs obtained from the proposed EIG-motivated surrogates outperform the Fisher-based designs.  

This first experiment shows that, when the prior is simple and the two
parameters have comparable observability, both proposed surrogates behave well
and improve over classical Fisher-based designs.

\paragraph{Uneven observability}

We next consider an asymmetric setting in which the observation-noise
standard deviations are set to \(\sigma_1=0.1\) and \(\sigma_2=0.12\). The corresponding error distributions are reported in Figure~\ref{fig:harmonic_uneven_obs_errors}.

\ErrorBoxplotFigure
  {csv/harmonic_uneven_obs_errors.csv}
  {Empirical distributions of the parameter-estimation errors for the harmonic oscillator test case with uneven observability over 1000 Monte Carlo runs. The box spans the interquartile range, whiskers show the 10th and 90th percentiles, the orange line denotes the median and the triangle
marks the mean. The instantaneous surrogate strongly favors the first component, whereas the tilting surrogate provides more balanced reconstructions.}
  {fig:harmonic_uneven_obs_errors}

Figure~\ref{fig:harmonic_uneven_obs_errors} reveals a clear limitation of the instantaneous surrogate in this regime. Because it is driven by immediate information gain only, it allocates  all measurements to the first sensor. This leads to highly accurate estimation of \(\theta_1\), but leaves \(\theta_2\) informed by the prior alone, which degrades the overall quality of parameter reconstruction.  By contrast, the tilting surrogate avoids this concentration effect, produces a more balanced sensing strategy, and still outperforms the Fisher-based designs in this example.

\subsection{Lotka-Volterra}

As a second benchmark problem, we consider a controlled Lotka-Volterra system over the time interval \([0,12]\), following the formulation proposed in \cite{Sager2006,Plate2026}. This example is intended to assess the behavior of the different design criteria under more complex prior distributions, and in particular to highlight the impact of prior multimodality on the resulting designs.

The state variable is defined by the two population levels,
\[
x(t)=\bigl(x_1(t),\,x_2(t)\bigr)^\top \in \mathbb{R}^2,
\]
and the dynamics are given by
\begin{equation}
\left\{
\begin{aligned}
\dot{x}_1(t) &= x_1(t) - \theta_1 x_1(t)x_2(t) - 0.4\,u(t)x_1(t),\\
\dot{x}_2(t) &= -x_2(t) + \theta_2 x_1(t)x_2(t) - 0.2\,u(t)x_2(t),
\end{aligned}
\right.
\end{equation}
where \(u(t)\in[0,1]\) is a piecewise-constant control input, assumed to be constant on each element of a uniform partition of \([0,12]\) into \(12\) subintervals, and \(\theta = (\theta_1,\theta_2)^\top \in \mathbb{R}_+^2\)
is the unknown parameter. The initial state is set to
\(x(0) = (0.5, 0.7)^\top.\)

Two sensors are available, each measuring the population of one species
\begin{equation}
y_1(t) = x_1(t) + \varepsilon_{1}(t)
\quad \text{and} \quad
y_2(t) = x_2(t) + \varepsilon_{2}(t),
\end{equation}
where \(\varepsilon_{1}(t)\) and \(\varepsilon_{2}(t)\) are centered Gaussian random variables, both with variance \(0.2\).

The designs are restricted to 10 distinct observation times, obtained by discretizing the relaxed optimal design with a minimum separation of 
0.25 time units. Each scalar Gaussian noise variable is discretized using a degree 4 Gauss–Hermite quadrature rule. For the Gaussian tilting surrogate and for the Fisher-based designs, the nominal parameter is chosen as the prior mean in each scenario. 

\paragraph{Log-normal prior}
We first consider a log-normal prior on the parameter \(\theta\), that is, with $\mathbf{1} = (1,1)^\top$,  $\log \theta \sim \mathcal{N}\bigl(\log(2)\,\mathbf{1},\,0.2\,I_2\bigr)$. Its discrete approximation is constructed using a degree-\(4\) Gauss--Hermite quadrature rule, with the quadrature nodes mapped through the exponential transformation. 

In this setting, the Fisher-based designs are already known to perform well, and only limited improvement is expected from the EIG-based surrogates. The empirical distributions of the estimation errors over \(1000\) Monte Carlo runs are shown in Figure~\ref{fig:LV_log_normal_errors}.

\ErrorBoxplotFigure
  {csv/LV_log_normal_errors.csv}
  {Empirical distributions of the parameter-estimation errors for the Lotka-Volterra test case with log-normal prior over 1000 Monte Carlo runs. The box spans the interquartile range, whiskers show the 10th and 90th percentiles, the orange line denotes the median and the triangle
marks the mean. All designs achieve comparable performance in this case, although the EIG-based surrogates remain slightly better overall.}
  {fig:LV_log_normal_errors}

As shown in Figure~\ref{fig:LV_log_normal_errors}, all designs perform similarly in this case, with only small differences between methods. This is expected under a narrow unimodal log-normal prior that remains close to a Gaussian distribution, so that a single nominal parameter already provides a reasonable local approximation. In this setting, the EIG-based surrogates bring only limited improvements over the Fisher-based designs.

\paragraph{Log-normal mixture prior}
We finally consider a more challenging prior given by an equal-weight mixture of two log-normal components, one associated with \(\mathcal{N}\bigl(\log 2\, \mathbf{1},\,0.2\,I_2\bigr).\) and the other with \(\mathcal{N}\bigl(\log 10 \,\mathbf{1},\,0.05\,I_2\bigr)\). Its discrete approximation is obtained by combining two degree-4 Gauss–Hermite quadrature rules, one for each mixture component. 

The main difficulty in this setting comes from the bimodal structure of the prior. For Fisher-based designs, which rely on a single nominal linearization point, the choice of that point becomes ambiguous and may lead to poor designs. By contrast, the instantaneous surrogate incorporates the multimodal structure directly through the prior discretization. The Gaussian tilting surrogate lies in between: it still depends on a single linearization point, but it accounts for prior multimodality. The corresponding error distributions are displayed in Figure~\ref{fig:LV_log_normal_mixt_errors}.

\ErrorBoxplotFigureWithSurrogateZoom
  {csv/LV_log_normal_mixt_errors.csv}
  {Empirical distributions of the parameter-estimation errors for the Lotka-Volterra test case with log-normal mixture prior over 1000 Monte Carlo runs. The box spans the interquartile range, whiskers show the 10th and 90th percentiles, the orange line denotes the median, and the triangle marks the mean. The Fisher-based designs deteriorate significantly in this bimodal setting, whereas the EIG-based surrogates remain effective and better adapt to the multimodal prior structure.}
  {fig:LV_log_normal_mixt_errors}
  {100}

Figure~\ref{fig:LV_log_normal_mixt_errors} shows a clear separation between the methods. The Fisher-based designs fail to produce reliable parameter estimates in this bimodal setting, whereas the EIG-based surrogates remain effective and generate designs adapted to the complex prior geometry. In particular, although the Gaussian tilting surrogate relies on the same poor nominal linearization point as the Fisher-based designs, it still leads to efficient parameter estimation. The proposed surrogates are thus more robust to prior complexity in this example than criteria derived solely from a local Fisher approximation.

\subsection{Comparison with the full-sequence EIG}

The previous experiments evaluated the various designs through downstream
estimation accuracy. We now compare the same designs directly on the Bayesian
objective by estimating the full-sequence EIG with a nested Monte Carlo
log-sum-exp estimator.  For each design, we draw \(N = 1{,}000\) prior samples and propagate
them through the dynamics. For each of \(S = 200\) noise replications,
the log-likelihood is accumulated over the full observation sequence,
and the mutual information is estimated via the standard nested
log-sum-exp over all particle pairs. 

Table~\ref{tab:cross_evaluation} reports the estimated full-sequence EIG and
the two surrogate values at each optimized design. Larger values indicate
better designs within each column. The gap between \(J_{\mathrm{inst}}\) and the EIG is partly due to the double-counting effect described in Lemma~4.1. In the examples considered, the Gaussian tilting surrogate reduces this gap. In all examples, one of the surrogate-optimized designs attains the
largest estimated EIG.

\begin{table}[htbp]
\centering
\caption{Surrogate and EIG values at each optimized design. Bold
indicates the highest EIG in each example.}
\label{tab:cross_evaluation}
\small
\begin{tabular}{ll rrr @{\qquad} ll rrr}
\toprule
\multicolumn{5}{c}{Harmonic oscillator}
& \multicolumn{5}{c}{Lotka--Volterra} \\
\cmidrule(r){1-5} \cmidrule(l){6-10}
Case & Design & $J_{\mathrm{inst}}$ & $J_{\mathrm{tilt}}$ & EIG
& Case & Design & $J_{\mathrm{inst}}$ & $J_{\mathrm{tilt}}$ & EIG \\
\midrule
\multirow{5}{*}{even}
  & A-opt     & \V\harmsim{1}{J_inst} & \V\harmsim{1}{J_tilt} & \Ve\harmsim{1}
  && A-opt     & \V\lvuni{1}{J_inst} & \V\lvuni{1}{J_tilt} & \Ve\lvuni{1} \\
  & Linearised
   & \V\harmsim{2}{J_inst} & \V\harmsim{2}{J_tilt} & \Ve\harmsim{2}
  && Linearised
   & \V\lvuni{2}{J_inst} & \V\lvuni{2}{J_tilt} & \Ve\lvuni{2} \\
  & EIG-inst  & \V\harmsim{3}{J_inst} & \V\harmsim{3}{J_tilt} & \Ve\harmsim{3}
  && EIG-inst  & \V\lvuni{3}{J_inst} & \V\lvuni{3}{J_tilt} & \Ve\lvuni{3} \\
  & EIG-tilt  & \V\harmsim{4}{J_inst} & \V\harmsim{4}{J_tilt} & \Ve\harmsim{4}
  && EIG-tilt  & \V\lvuni{4}{J_inst} & \V\lvuni{4}{J_tilt} & \Ve\lvuni{4} \\
\midrule
\multirow{5}{*}{uneven}
  & A-opt     & \V\harmunev{1}{J_inst} & \V\harmunev{1}{J_tilt} & \Ve\harmunev{1}
  && A-opt     & \V\lvbi{1}{J_inst} & \V\lvbi{1}{J_tilt} & \Ve\lvbi{1} \\
  & Linearised
   & \V\harmunev{2}{J_inst} & \V\harmunev{2}{J_tilt} & \Ve\harmunev{2}
  && Linearised
   & \V\lvbi{2}{J_inst} & \V\lvbi{2}{J_tilt} & \Ve\lvbi{2} \\
  & EIG-inst  & \V\harmunev{3}{J_inst} & \V\harmunev{3}{J_tilt} & \Ve\harmunev{3}
  && EIG-inst  & \V\lvbi{3}{J_inst} & \V\lvbi{3}{J_tilt} & \Ve\lvbi{3} \\
  & EIG-tilt  & \V\harmunev{4}{J_inst} & \V\harmunev{4}{J_tilt} & \Ve\harmunev{4}
  && EIG-tilt  & \V\lvbi{4}{J_inst} & \V\lvbi{4}{J_tilt} & \Ve\lvbi{4} \\
\bottomrule
\end{tabular}
\end{table}

\subsection{Computational cost}
\label{sec:cost}

We report the computational workload of the main criteria in
Table~\ref{tab:complexity}. The table separates the dimension of the
augmented ODE system propagated over \(N_t\) time steps from the dominant
likelihood and log-sum-exp operations applied to the resulting trajectories.
These counts should be interpreted as leading-order workload estimates, up to
model-dependent constants.

\begin{table}[htbp]
\centering
\caption{Computational workload per objective evaluation. The first column
gives the dimension of the augmented ODE system propagated over \(N_t\) time
steps. The second column gives the dominant cost of likelihood and
log-sum-exp operations applied after trajectory propagation.}
\label{tab:complexity}
\small
\begin{tabular}{l l l}
\toprule
Criterion & Augmented ODE dimension & Dominant likelihood-sum cost \\
\midrule
A/D-optimal
  & \(n_x + n_x n_\theta + \frac{n_\theta(n_\theta+1)}{2} + 1\)
  & --- \\[0.8ex]
Linearised
  & \(n_x + n_x n_\theta + \frac{n_\theta(n_\theta+1)}{2} + 1\)
  & \(O(n_\theta^3)\) terminal log-determinant \\[0.8ex]
\(J_{\mathrm{inst}}\)
  & \(N n_x + 1\)
  & \(O(N_t\,M_{\mathrm{cfg}}\,Q\,N^2)\) \\[0.8ex]
\(J_{\mathrm{tilt}}\)
  & \(N n_x + n_x + n_x n_\theta
     + \frac{n_\theta(n_\theta+1)}{2} + N + 1\)
  & \(O(N_t\,M_{\mathrm{cfg}}\,Q\,N^2)
     + O(N_t\,N\,n_\theta^2)\) \\
\bottomrule
\end{tabular}
\end{table}

Here \(N_t\) denotes the number of time steps, \(N\) the number of prior
nodes, \(Q\) the number of deterministic noise quadrature nodes per active
observation block, and \(M_{\mathrm{cfg}}\) the number of non-empty sensor
configurations evaluated at each time step
(\(M_{\mathrm{cfg}}=2^{n_{\exp}}-1\) in the present formulation).
In terms of forward-model evaluations, A/D-optimality and the linearised
criterion require one nominal trajectory, \(J_{\mathrm{inst}}\) requires
\(N\) parameter trajectories, and \(J_{\mathrm{tilt}}\) requires these same
\(N\) trajectories plus one reference trajectory. The trajectories are reused
across all \(Q\) noise nodes and likelihood sums. Gradients are computed by
reverse-mode automatic differentiation of the same computational graph and
therefore have the same leading-order scaling, up to a constant factor.

Figure~\ref{fig:scaling} reports build and solve times, IPOPT iteration counts,
and surrogate values at the optimum as functions of \(N\) and \(Q\) on the
Lotka--Volterra unimodal example. All computations were performed on a MacBook Pro equipped with an
Apple M4 Pro processor and 48~GB of RAM. The timings as \(N\) varies are consistent
with the \(O(N^2)\) pairwise likelihood sums reported in
Table~\ref{tab:complexity}. The dependence on \(Q\) is less clear over the
tested range, although no contradiction with the expected linear quadrature
cost is observed. The IPOPT iteration counts fluctuate and are not used to
infer scaling. The surrogate values stabilize around the quadrature settings
used in the main experiments, \(N=36\) and \(Q=25\).

\pgfplotsset{
  scalingcommon/.style={
    width=0.44\textwidth,
    height=0.24\textwidth,
    tick pos=left,
    grid=none,
    axis line style={black, line width=0.4pt},
    tick style={black, line width=0.4pt},
    tick label style={font=\small},
    label style={font=\small},
    title style={font=\small},
    every axis plot/.append style={
      line width=0.8pt,
      mark size=1.6pt
    },
    legend style={
      font=\scriptsize,
      draw=none,
      fill=none,
      cells={anchor=west},
      /tikz/every even column/.append style={column sep=0.35cm}
    },
    legend columns=5
  },
  instsolve/.style={
    color=blue!70!black,
    solid,
    mark=*
  },
  instbuild/.style={
    color=blue!70!black,
    densely dashed,
    mark=o
  },
  tiltsolve/.style={
    color=orange!85!black,
    solid,
    mark=triangle*
  },
  tiltbuild/.style={
    color=orange!85!black,
    densely dashed,
    mark=triangle*
  },
  refline/.style={
    color=black!45,
    dashed,
    thin,
    mark=none
  }
}

\begin{figure}[htbp]
\centering

\begin{tikzpicture}
\begin{groupplot}[
  group style={
    group size=2 by 3,
    horizontal sep=1.35cm,
    vertical sep=1.35cm
  },
  scalingcommon
]


\nextgroupplot[
  title={Varying \(N\)},
  ylabel={Time (s)},
  xmode=log,
  ymode=log,
  legend to name=combinedlegend
]
\addplot[instbuild] table[x=N, y=build_s] {\Ninst};
\addlegendentry{\(J_{\mathrm{inst}}\) build}
\addplot[instsolve] table[x=N, y=solve_s] {\Ninst};
\addlegendentry{\(J_{\mathrm{inst}}\) solve}
\addplot[tiltbuild] table[x=N, y=build_s] {\Ntilt};
\addlegendentry{\(J_{\mathrm{tilt}}\) build}
\addplot[tiltsolve] table[x=N, y=solve_s] {\Ntilt};
\addlegendentry{\(J_{\mathrm{tilt}}\) solve}
\addplot[refline] table[x=N, y=ref_solve] {\Ninst};
\addlegendentry{reference slope}

\nextgroupplot[
  title={Varying \(Q\)},
  xmode=log,
  ymode=log
]
\addplot[instbuild] table[x=Q, y=build_s] {\Qinst};
\addplot[instsolve] table[x=Q, y=solve_s] {\Qinst};
\addplot[tiltbuild] table[x=Q, y=build_s] {\Qtilt};
\addplot[tiltsolve] table[x=Q, y=solve_s] {\Qtilt};
\addplot[refline] table[x=Q, y=ref_solve] {\Qinst};


\nextgroupplot[
  ylabel={IPOPT iterations}
]
\addplot[instsolve] table[x=N, y=n_iter] {\Ninst};
\addplot[tiltsolve] table[x=N, y=n_iter] {\Ntilt};

\nextgroupplot
\addplot[instsolve] table[x=Q, y=n_iter] {\Qinst};
\addplot[tiltsolve] table[x=Q, y=n_iter] {\Qtilt};


\nextgroupplot[
  xlabel={\(N\)},
  ylabel={Value at optimum}
]
\addplot[instsolve] table[x=N, y=J_value] {\Ninst};
\addplot[tiltsolve] table[x=N, y=J_value] {\Ntilt};

\nextgroupplot[
  xlabel={\(Q\)}
]
\addplot[instsolve] table[x=Q, y=J_value] {\Qinst};
\addplot[tiltsolve] table[x=Q, y=J_value] {\Qtilt};

\end{groupplot}
\end{tikzpicture}

\vspace{0.15cm}

\pgfplotslegendfromname{combinedlegend}

\vspace{0.1cm}

\caption{
Scaling of build and solve time, IPOPT iteration count, and surrogate value at
the optimum as a function of the number of prior nodes~\(N\) (left column,
\(Q=25\) fixed) and noise quadrature nodes~\(Q\) (right column, \(N=36\) fixed),
for \(J_{\mathrm{inst}}\) and \(J_{\mathrm{tilt}}\) on the Lotka--Volterra
unimodal example. The gray dashed lines show \(O(N^2)\) and \(O(Q)\) reference
slopes. The \(A\)-optimal and local linear-Gaussian EIG baselines are independent of
\(N\) and \(Q\) and are omitted for clarity.
}
\label{fig:scaling}
\end{figure}

%% file: sections/conclusion.tex
In this paper, we introduced two surrogates of the expected information gain that are compatible with efficient adjoint-based optimization methods. Our goal was to move beyond Fisher-information-based design criteria in controlled dynamical systems while retaining enough structure to make optimization tractable.

The proposed surrogates rely on a chain-rule decomposition of the expected information gain together with tractable approximations of the posterior distribution of the unknown parameter given past observations. In the instantaneous surrogate, this posterior is replaced by the prior, whereas in the Gaussian tilting surrogate it is approximated through a Fisher-driven Gaussian tilting. 

On the theoretical side, we showed that the instantaneous surrogate may double count information, although it can still approximate the expected information gain well in favorable regimes. The Gaussian tilting surrogate enjoys stronger consistency properties and is exact in the linear-Gaussian setting. 

Our numerical experiments illustrate both the strengths and the limitations of the proposed criteria. In relatively simple or nearly Gaussian settings, the gains over Fisher-based designs remain modest. In contrast, for non-Gaussian or multimodal priors, the EIG-motivated surrogates yield clearer improvements. The experiments also highlight a limitation of the instantaneous surrogate, which may allocate too much sensing effort to the easiest-to-observe components, whereas the tilting surrogate produces more balanced sensing strategies. 

One direction for improvement is to better control weight degeneracy in the tilting approximation. For instance, one could temper the tilting surrogate by introducing a parameter $\alpha \in (0,1]$ in the quadratic exponent, thereby slowing down the concentration of the surrogate weights. Such a parameter could be selected adaptively, for instance using effective sample size criteria inspired by particle filtering \cite{Douc2005,Li2015,lagracie2025}.

As with most EIG-based approaches, the main challenge remains scalability in high dimension. A natural next step is therefore to combine the proposed methods with dimension-reduction techniques such as truncated SVD or lumping \cite{Lopez2015,Benner2021,Plate2026}. Moreover, the optimization problem considered here relies on a discretized representation of the prior. While this is sufficient for the purposes of the present paper, improving the scalability of this prior representation remains an important direction for future work. A further step would be to remove the need for a fixed discrete prior approximation altogether, for instance through variational families \cite{Dong2025} or transport maps \cite{Koval2024}.

%% file: sections/appendix.tex
This annex contains the two technical lemmas required for the proof of Proposition~\ref{prop:consistency_LG}.

\begin{lemma} \label{lmm:weak_conv_qk}
Let $(p_k)_{k\ge1}$ be a sequence of probability measures on $\mathbb{R}^d$  with mean $m_k$ and finite second moment such that
\[
p_k \to p_0 := \mathcal N(m_0,\Sigma_0)
\quad\text{in } W_2.
\]

Let $F$ be a symmetric positive semidefinite matrix, and consider the probability measures 

\[
q^{(k)}(d\theta)
=
\frac{\phi^{(k)}(\theta)}{Z^{(k)}}\,p_k(d\theta)\quad\text{and}\quad
q^{(0)}(d\theta)
=
\frac{\phi(\theta)}{Z^{(0)}}\,p_0(d\theta)
\]

where
\[
\phi^{(k)}(\theta)
=
\exp\!\left(-\tfrac12(\theta-m_k)^\top F (\theta-m_k)\right),
\qquad
\phi(\theta)
=
\exp\!\left(-\tfrac12(\theta-m_0)^\top F (\theta-m_0)\right),
\]
and
\[
Z^{(k)}=\int \phi^{(k)}(\theta)p_k(d\theta),
\qquad
Z^{(0)}=\int \phi(\theta)p_0(d\theta).
\]

Let $(f_k)$ be continuous functions converging locally uniformly to $f$ and satisfying
\[
|f_k(\theta)| \le C(1+\|\theta\|^2).
\]

Then
\[
\mathbb{E}_{q^{(k)}}[f_k(\theta)]
\;\longrightarrow\;
\mathbb{E}_{q^{(0)}}[f(\theta)] .
\]
\end{lemma}

\begin{proof}
Convergence in $W_2$ is equivalent to weak convergence together with convergence of second moments \cite{Villani2009}. In particular $m_k\to m_0$ which implies that $\phi^{(k)}\to\phi$ locally uniformly.

Let $(f_k)$ be a sequence of continuous functions converging locally
uniformly to $f$ and satisfying the uniform quadratic bound
\(
|f_k(\theta)| \le C(1+\|\theta\|^2)
\) for all $k$.

Since $\phi^{(k)}\le1$, the functions $f_k\phi^{(k)}$ also satisfy a
quadratic bound and therefore form a uniformly integrable family under
the uniform second-moment bound on $(p_k)$.

We write
\begin{align}\label{eq:decomp_convergence}
\int f_k(\theta)\phi^{(k)}(\theta)\,p_k(d\theta)
&=
\int f_k(\theta)\big(\phi^{(k)}(\theta)-\phi(\theta)\big)\,p_k(d\theta)
\\
&\quad+
\int (f_k(\theta)-f(\theta))\phi(\theta)\,p_k(d\theta)
\\
&\quad+
\int f(\theta)\phi(\theta)\,p_k(d\theta).
\end{align}

The first two terms converge to zero by the local uniform convergence
$\phi^{(k)}\to\phi$ and $f_k\to f$, combined with the quadratic
growth bound and the uniform second-moment bound on $(p_k)$,
which ensures uniform integrability.
The third term converges to
\(
\int f(\theta)\phi(\theta)\,p_0(d\theta)
\)
by the weak convergence of $p_k$ implied by $W_2$ convergence.

In particular, this yields the convergence of the normalization constants
\[
Z^{(k)} := \int \phi^{(k)}(\theta)\,p_k(d\theta)
\;\longrightarrow\;
Z^{(0)} := \int \phi(\theta)\,p_0(d\theta) > 0.
\]

Moreover, since
\(
q^{(k)}(d\theta)
=
\frac{\phi^{(k)}(\theta)}{Z^{(k)}}\,p_k(d\theta)\) and 
\(
q^{(0)}(d\theta)
=
\frac{\phi(\theta)}{Z^{(0)}}\,p_0(d\theta),
\)
we obtain
\begin{equation}
\mathbb{E}_{q^{(k)}}[f_k(\theta)]
\;\longrightarrow\;
\mathbb{E}_{q^{(0)}}[f(\theta)]. 
\end{equation}

\end{proof}

\begin{lemma}\label{lmm:behavior_denom_eig_gauss}
Let $(p_k)_{k\ge1}$, $q^{(k)}$, and $q^{(0)}$ be as in Lemma~\ref{lmm:weak_conv_qk}.
Let \(\ell(y\mid\theta)\) be the likelihood associated with the linear-Gaussian observation model
\[
y_d = H_{d}\theta + b_{d} + \varepsilon_d,
\qquad
\varepsilon_d\sim\mathcal N(0,R_{d}),
\qquad d=1,\dots,n_{\mathrm{exp}},
\]
with positive definite \(R_{d}\).

Define
\[
m_k(y):=\int \ell(y\mid\theta')\,q^{(k)}(d\theta'),
\qquad
m(y):=\int \ell(y\mid\theta')\,q^{(0)}(d\theta'),
\]
and
\[
g_k(\theta):=
-\mathbb{E}_{Y\sim \ell(\cdot\mid\theta)}[\log m_k(Y)],
\qquad
g(\theta):=
-\mathbb{E}_{Y\sim \ell(\cdot\mid\theta)}[\log m(Y)].
\]

Then \(g_k\to g\) locally uniformly on \(\mathbb R^{n_\theta}\), and there exists \(C>0\) such that
\[
|g_k(\theta)|\le C(1+\|\theta\|^2)
\qquad\text{for all }k,\theta.
\]
\end{lemma}

\begin{proof}
Let \(\Theta'\sim q^{(k)}\). Then
\[
m_k(y)
=
c\,\E\!\left[
\exp\!\left(
-\tfrac12
\sum_{d=1}^{n_{\mathrm{exp}}}
\|y_d-H_d\Theta'-b_d\|_{R_d^{-1}}^2
\right)
\right].
\]

Since \(0<m_k(y)\le c\), it remains to control \(-\log m_k(y)\). Applying Lemma~\ref{lmm:weak_conv_qk} with \(f_k(\theta)=\|\theta\|^2\), we obtain \(\mathbb{E}_{q^{(k)}}[\|\theta\|^2]\to \mathbb{E}_{q^{(0)}}[\|\theta\|^2]\), hence \(\bigl(q^{(k)}\bigr)\) has uniformly bounded second moments. Combined with Jensen's inequality, this gives
\[
-\log m_k(y)
\le
C\,
\E\!\left[
1+\|y\|^2+\|\Theta'\|^2
\right]
\le
C(1+\|y\|^2),
\]

uniformly in \(k\). Hence

\begin{equation}\label{eq:quad_bound_mk}
|\log m_k(y)|\le C(1+\|y\|^2),
\end{equation}
uniformly in \(k\). Taking expectation with respect to \(Y\sim \ell(\cdot\mid\theta)\),
\[
|g_k(\theta)|
\le \E[|\log m_k(Y)|]
\le C\bigl(1+\E[\|Y\|^2]\bigr)
\le C(1+\|\theta\|^2),
\]
since \(Y\mid\theta\) is Gaussian with second moment bounded by \(C(1+\|\theta\|^2)\).

Next, Lemma~\ref{lmm:weak_conv_qk} gives that \(q^{(k)}\) converges weakly toward \(q^{(0)}\). For each fixed \(y\), the map
\(\theta'\mapsto \ell(y\mid\theta')\) is bounded and continuous, so
\(
m_k(y)\to m(y).
\)

We next prove that the convergence is locally uniform in \(y\). Indeed, \(y\mapsto \ell(y\mid\theta')\) is \(C^1\), with
\[
\nabla_y \ell(y\mid\theta')=A(y,\theta')\,\ell(y\mid\theta'),
\]
where \(A(y,\theta')\) is affine in \((y,\theta')\). Thus, for every compact \(K\) in the observation space,
\[
\sup_{y\in K}\|\nabla_y \ell(y\mid\theta')\|\le C_K(1+\|\theta'\|).
\]
Using the uniform second-moment bound on \(q^{(k)}\),
\[
\sup_k\int \sup_{y\in K}\|\nabla_y \ell(y\mid\theta')\|\,q^{(k)}(d\theta')<\infty,
\]
so \((m_k)_k\) is equi-Lipschitz on \(K\). Combined with pointwise convergence, this yields
\(m_k\to m\) locally uniformly in \(y\).

Since \(m\) is continuous and strictly positive, every compact \(K\) admits \(\delta_K>0\) such that
\(m\ge \delta_K\) on \(K\). By local uniform convergence, \(m_k\ge \delta_K/2\) on \(K\) for all large \(k\), hence
\[
\log m_k\to \log m
\qquad\text{locally uniformly in } y.
\]

Finally, let \(K\subset\mathbb R^{n_\theta}\) be compact and \(Y\sim \ell(\cdot\mid\theta)\). Then
\[
|g_k(\theta)-g(\theta)|
\le
\E_{\ell(\cdot\mid\theta)}
\bigl[|\log m_k(Y)-\log m(Y)|\bigr].
\]
For \(R>0\), we split the expected value according to \(\{\|Y\|\le R\}\cup\{\|Y\|>R\}\). On \(\{\|Y\|\le R\}\), the integrand converges uniformly in \(y\), hence
\[
\sup_{\theta\in K}
\E\!\left[
|\log m_k(Y)-\log m(Y)|\,\mathbf 1_{\{\|Y\|\le R\}}
\right]\to 0.
\]
On \(\{\|Y\|>R\}\), \eqref{eq:quad_bound_mk} gives
\[
|\log m_k(y)-\log m(y)|\le C(1+\|y\|^2),
\]
so
\[
\sup_{\theta\in K}
\E_{\ell(\cdot\mid\theta)}
\!\left[
|\log m_k(Y)-\log m(Y)|\,\mathbf 1_{\{\|Y\|>R\}}
\right]
\le
C\sup_{\theta\in K}
\E_{\ell(\cdot\mid\theta)}
\!\left[(1+\|Y\|^2)\mathbf 1_{\{\|Y\|>R\}}\right].
\]
Because \(Y\mid\theta\) is Gaussian with mean affine in \(\theta\) and covariance independent of \(\theta\), the right-hand side tends to \(0\) as \(R\to\infty\), uniformly in \(\theta\in K\). Therefore
\[
\sup_{\theta\in K}|g_k(\theta)-g(\theta)|\to0,
\]
so \(g_k\to g\) locally uniformly in \(\theta\).
\end{proof}

%% file: main.bib
@article{Waechter2006,
  author = {W{\"a}chter, Andreas and Biegler, Lorenz T.},
  title = {On the implementation of an interior-point filter line-search algorithm for large-scale nonlinear programming},
  journal = {Mathematical Programming},
  volume = {106},
  pages = {25--57},
  year = {2006},
  doi = {10.1007/s10107-004-0559-y}
}

@misc{li2024logsobolev,
  author = {Li, Fengyi and Belhadji, Ayoub and Marzouk, Youssef},
  title = {Nonlinear {B}ayesian optimal experimental design using logarithmic {S}obolev inequalities},
  year = {2024},
  doi = {10.48550/arXiv.2402.15053}
}

@article{qi2025observationnoise,
  author = {Qi, Jie and Baker, Ruth E.},
  title = {Optimal experimental design for parameter estimation in the presence of observation noise},
  journal = {Mathematical Biosciences},
  volume = {392},
  pages = {109571},
  year = {2026},
  doi = {10.1016/j.mbs.2025.109571}
}

@article{dick2013high,
  author = {Dick, Josef and Kuo, Frances Y. and Sloan, Ian H.},
  title = {High-dimensional integration: The quasi-{Monte} {Carlo} way},
  journal = {Acta Numerica},
  volume = {22},
  pages = {133--288},
  year = {2013},
  doi = {10.1017/S0962492913000044}
}

@article{bungartz2004sparse,
  author = {Bungartz, Hans-Joachim and Griebel, Michael},
  title = {Sparse grids},
  journal = {Acta Numerica},
  volume = {13},
  pages = {147--269},
  year = {2004},
  doi = {10.1017/S0962492904000182}
}

@article{mentre1997optimal,
  author = {Mentr{\'e}, France and Mallet, Alain and Baccar, Doha},
  title = {Optimal design in random-effects regression models},
  journal = {Biometrika},
  volume = {84},
  pages = {429--442},
  year = {1997},
  doi = {10.1093/biomet/84.2.429}
}

@article{korkel2004numerical,
  author = {K{\"o}rkel, Stefan and Kostina, Ekaterina and Bock, Hans Georg and Schl{\"o}der, Johannes P.},
  title = {Numerical methods for optimal control problems in design of robust optimal experiments for nonlinear dynamic processes},
  journal = {Optimization Methods and Software},
  volume = {19},
  pages = {327--338},
  year = {2004},
  doi = {10.1080/10556780410001683078}
}

@article{pagendam2013optimal,
  author = {Pagendam, Dan E. and Pollett, Philip K.},
  title = {Optimal design of experimental epidemics},
  journal = {Journal of Statistical Planning and Inference},
  volume = {143},
  pages = {563--572},
  year = {2013},
  doi = {10.1016/j.jspi.2012.09.011}
}

@misc{maio2025,
  author = {Maio, Steven and Alexanderian, Alen},
  title = {On submodularity of the expected information gain},
  year = {2025},
  doi = {10.48550/arXiv.2505.04145}
}

@inproceedings{Rainforth2018,
  author = {Rainforth, Tom and Cornish, Rob and Yang, Hongseok and Warrington, Andrew and Wood, Frank},
  title = {On Nesting {M}onte {C}arlo Estimators},
  booktitle = {Proceedings of the 35th International Conference on Machine Learning},
  pages = {4267--4276},
  year = {2018},
  doi = {10.48550/arXiv.1709.06181}
}

@article{Paulson2019,
  author = {Paulson, Joel and Martin-Casas, Marc and Mesbah, Ali},
  title = {Optimal Bayesian experiment design for nonlinear dynamic systems with chance constraints},
  journal = {Journal of Process Control},
  volume = {77},
  pages = {155--171},
  year = {2019},
  doi = {10.1016/j.jprocont.2019.01.010}
}

@article{Overstall2019,
  author = {Overstall, Antony and Woods, David and Parker, Ben},
  title = {Bayesian Optimal Design for Ordinary Differential Equation Models With Application in Biological Science},
  journal = {Journal of the American Statistical Association},
  volume = {115},
  pages = {583--598},
  year = {2020},
  doi = {10.1080/01621459.2019.1617154}
}

@inproceedings{Busetto2009,
  author = {Busetto, Alberto Giovanni and Ong, Cheng Soon and Buhmann, Joachim},
  title = {Optimized expected information gain for nonlinear dynamical systems},
  booktitle = {Proceedings of the 26th International Conference on Machine Learning},
  pages = {97--104},
  year = {2009},
  doi = {10.1145/1553374.1553387}
}

@misc{Foster2019,
  author = {Foster, Adam and Jankowiak, Martin and Bingham, Eli and Horsfall, Paul and Teh, Yee Whye and Rainforth, Tom and Goodman, Noah},
  title = {Variational Estimators for Bayesian Optimal Experimental Design},
  year = {2019},
  doi = {10.48550/arXiv.1903.05480}
}

@article{Korkel2004,
  author = {K{\"o}rkel, Stefan and Kostina, Ekaterina and Bock, Hans Georg and Schl{\"o}der, Johannes P.},
  title = {Numerical methods for optimal control problems in design of robust optimal experiments for nonlinear dynamic processes},
  journal = {Optimization Methods and Software},
  volume = {19},
  pages = {327--338},
  year = {2004},
  doi = {10.1080/10556780410001683078}
}

@article{Franceschini2008,
  author = {Franceschini, Gaia and Macchietto, Sandro},
  title = {Model-based design of experiments for parameter precision: State of the art},
  journal = {Chemical Engineering Science},
  volume = {63},
  pages = {4846--4872},
  year = {2008},
  doi = {10.1016/j.ces.2007.11.034}
}

@article{Kreutz2009,
  author = {Kreutz, Clemens and Timmer, Jens},
  title = {Systems biology: experimental design},
  journal = {The FEBS Journal},
  volume = {276},
  pages = {923--942},
  year = {2009},
  doi = {10.1111/j.1742-4658.2008.06843.x}
}

@book{Martijn2005,
  editor = {Berger, Martijn P. F. and Wong, Weng Kee},
  title = {Applied Optimal Designs},
  publisher = {Wiley},
  year = {2005},
  doi = {10.1002/0470857005}
}

@article{Stone1959,
  author = {Stone, M.},
  title = {Application of a Measure of Information to the Design and Comparison of Regression Experiments},
  journal = {The Annals of Mathematical Statistics},
  volume = {30},
  pages = {55--70},
  year = {1959},
  doi = {10.1214/aoms/1177706359}
}

@article{Kiefer1958,
  author = {Kiefer, J.},
  title = {On the Nonrandomized Optimality and Randomized Nonoptimality of Symmetrical Designs},
  journal = {The Annals of Mathematical Statistics},
  volume = {29},
  pages = {675--699},
  year = {1958},
  doi = {10.1214/aoms/1177706530}
}

@article{Lindley1956,
  author = {Lindley, D. V.},
  title = {On a Measure of the Information Provided by an Experiment},
  journal = {The Annals of Mathematical Statistics},
  volume = {27},
  pages = {986--1005},
  year = {1956},
  doi = {10.1214/aoms/1177728069}
}

@article{Dong2025,
  author = {Dong, Jiayuan and Jacobsen, Christian and Khalloufi, Mehdi and Akram, Maryam and Liu, Wanjiao and Duraisamy, Karthik and Huan, Xun},
  title = {Variational Bayesian optimal experimental design with normalizing flows},
  journal = {Computer Methods in Applied Mechanics and Engineering},
  volume = {433},
  pages = {117457},
  year = {2025},
  doi = {10.1016/j.cma.2024.117457}
}

@article{Koval2024,
  author = {Koval, Karina and Herzog, Roland and Scheichl, Robert},
  title = {Tractable optimal experimental design using transport maps},
  journal = {Inverse Problems},
  volume = {40},
  pages = {125002},
  year = {2024},
  doi = {10.1088/1361-6420/ad8260}
}

@article{lagracie2025,
author = {Lagracie, Emma and de Montella, Luc},
title = {Particle Filtering for Nondeterministic Electrocardiographic Imaging},
journal = {SIAM Journal on Imaging Sciences},
volume = {19},
number = {3},
pages = {1409-1438},
year = {2026},
doi = {10.1137/25M1795492},
}

@article{Li2015,
  author = {Li, Tian-cheng and Villarrubia, Gabriel and Sun, Shu-dong and Corchado, Juan M. and Bajo, Javier},
  title = {Resampling methods for particle filtering: identical distribution, a new method, and comparable study},
  journal = {Frontiers of Information Technology \& Electronic Engineering},
  volume = {16},
  pages = {969--984},
  year = {2015},
  doi = {10.1631/FITEE.1500199}
}

@inproceedings{Douc2005,
  author = {Douc, R. and Capp{\'e}, O.},
  title = {Comparison of resampling schemes for particle filtering},
  booktitle = {Proceedings of the 4th International Symposium on Image and Signal Processing and Analysis},
  pages = {64--69},
  year = {2005},
  doi = {10.1109/ISPA.2005.195385}
}

@incollection{Benner2021,
  author = {Benner, Peter and Grivet-Talocia, Stefano and Quarteroni, Alfio and Rozza, Gianluigi and Schilders, Wil and Silveira, Luis},
  title = {Model order reduction: basic concepts and notation},
  booktitle = {System- and Data-Driven Methods and Algorithms},
  publisher = {De Gruyter},
  pages = {1--14},
  year = {2021},
  doi = {10.1515/9783110498967-001}
}

@article{Lopez2015,
  author = {L{\'o}pez C., Diana C. and Barz, Tilman and K{\"o}rkel, Stefan and Wozny, G{\"u}nter},
  title = {Nonlinear ill-posed problem analysis in model-based parameter estimation and experimental design},
  journal = {Computers \& Chemical Engineering},
  volume = {77},
  pages = {24--42},
  year = {2015},
  doi = {10.1016/j.compchemeng.2015.03.002}
}

@article{Plate2026,
  author = {Plate, Christoph and Martensen, Carl Julius and Sager, Sebastian},
  title = {Optimal Experimental Design for Universal Differential Equations},
  journal = {IEEE Transactions on Automatic Control},
  volume = {71},
  pages = {1521--1536},
  year = {2026},
  doi = {10.1109/TAC.2025.3609533}
}

@inproceedings{Sager2006,
  author = {Sager, Sebastian and Bock, Hans Georg and Diehl, Moritz and Reinelt, Gerhard and Schl{\"o}der, Johannes P.},
  title = {Numerical Methods for Optimal Control with Binary Control Functions Applied to a Lotka-Volterra Type Fishing Problem},
  booktitle = {Recent Advances in Optimization},
  pages = {269--289},
  year = {2006},
  doi = {10.1007/3-540-28258-0_17}
}

@article{Andersson2019,
  author = {Andersson, Joel A. E. and Gillis, Joris and Horn, Greg and Rawlings, James B. and Diehl, Moritz},
  title = {{CasADi} -- A software framework for nonlinear optimization and optimal control},
  journal = {Mathematical Programming Computation},
  volume = {11},
  pages = {1--36},
  year = {2019},
  doi = {10.1007/s12532-018-0139-4}
}

@article{LANEVE2010,
  author = {Laneve, Cosimo and Lascu, Tudor A. and Sordoni, Vania},
  title = {The Interval Analysis of Multilinear Expressions},
  journal = {Electronic Notes in Theoretical Computer Science},
  volume = {267},
  pages = {43--53},
  year = {2010},
  doi = {10.1016/j.entcs.2010.09.017}
}

@article{Huan_Jagalur_Marzouk_2024,
  author = {Huan, Xun and Jagalur, Jayanth and Marzouk, Youssef},
  title = {Optimal experimental design: Formulations and computations},
  journal = {Acta Numerica},
  volume = {33},
  pages = {715--840},
  year = {2024},
  doi = {10.1017/S0962492924000023}
}

@article{Alexanderian2014,
  author = {Alexanderian, Alen and Petra, Noemi and Stadler, Georg and Ghattas, Omar},
  title = {A-Optimal Design of Experiments for Infinite-Dimensional Bayesian Linear Inverse Problems with Regularized l0-Sparsification},
  journal = {SIAM Journal on Scientific Computing},
  volume = {36},
  pages = {A2122--A2148},
  year = {2014},
  doi = {10.1137/130933381}
}

@article{Haber2008,
  author = {Haber, Eldad and Horesh, Lior and Tenorio, Luis},
  title = {Numerical Methods for Experimental Design of Large-Scale Linear Ill-Posed Inverse Problems},
  journal = {Inverse Problems},
  volume = {24},
  pages = {055012},
  year = {2008},
  doi = {10.1088/0266-5611/24/5/055012}
}

@inproceedings{Zhou2024,
  author = {Zhou, Kangjie and Wu, Pengying and Su, Yao and Gao, Han and Ma, Ji and Liu, Hangxin and Liu, Chang},
  title = {{ASPIRe}: An Informative Trajectory Planner with Mutual Information Approximation for Target Search and Tracking},
  booktitle = {2024 IEEE International Conference on Robotics and Automation (ICRA)},
  pages = {4626--4632},
  year = {2024},
  doi = {10.1109/ICRA57147.2024.10611500}
}

@book{Villani2009,
  author = {Villani, C{\'e}dric},
  title = {Optimal Transport: Old and New},
  publisher = {Springer},
  year = {2009},
  doi = {10.1007/978-3-540-71050-9}
}

@article{Lewi2009,
  author = {Lewi, Jakub and Butera, Robert J. and Paninski, Liam},
  title = {Sequential Optimal Design of Neurophysiology Experiments},
  journal = {Neural Computation},
  volume = {21},
  pages = {619--687},
  year = {2009},
  doi = {10.1162/neco.2008.08-07-594}
}

@article{Long2013,
  author = {Long, Quan and Scavino, Marco and Tempone, Ra{\'u}l and Wang, Suojin},
  title = {Fast estimation of expected information gains for Bayesian experimental designs based on Laplace approximations},
  journal = {Computer Methods in Applied Mechanics and Engineering},
  volume = {259},
  pages = {24--39},
  year = {2013},
  doi = {10.1016/j.cma.2013.02.017}
}

@article{Drovandi2014,
  author = {Drovandi, Christopher C. and McGree, James M. and Pettitt, Anthony N.},
  title = {A Sequential Monte Carlo Algorithm to Incorporate Model Uncertainty in Bayesian Sequential Design},
  journal = {Journal of Computational and Graphical Statistics},
  volume = {23},
  pages = {3--24},
  year = {2014},
  doi = {10.1080/10618600.2012.730083}
}

@article{Cavagnaro2010,
  author = {Cavagnaro, Daniel R. and Myung, Jay I. and Pitt, Mark A. and Kujala, Janne V.},
  title = {Adaptive Design Optimization: A Mutual Information-Based Approach to Model Discrimination in Cognitive Science},
  journal = {Neural Computation},
  volume = {22},
  pages = {887--905},
  year = {2010},
  doi = {10.1162/neco.2009.02-09-959}
}

@article{Krause2008,
  author = {Krause, Andreas and Singh, Ajit and Guestrin, Carlos},
  title = {Near-Optimal Sensor Placements in Gaussian Processes: Theory, Efficient Algorithms and Empirical Studies},
  journal = {Journal of Machine Learning Research},
  volume = {9},
  pages = {235--284},
  year = {2008}
}

@article{Ryan2003,
  author = {Ryan, Kenneth J.},
  title = {Estimating Expected Information Gains for Experimental Designs with Application to the Random Fatigue-Limit Model},
  journal = {Journal of Computational and Graphical Statistics},
  volume = {12},
  pages = {585--603},
  year = {2003},
  doi = {10.1198/1061860032012}
}

@book{Lindley1972,
  author = {Lindley, Dennis V.},
  title = {Bayesian Statistics: A Review},
  publisher = {SIAM},
  year = {1972},
  doi = {10.1137/1.9781611970654}
}

@book{RaiffaSchlaifer1961,
  author = {Raiffa, Howard and Schlaifer, Robert},
  title = {Applied Statistical Decision Theory},
  publisher = {Harvard University Press},
  year = {1961}
}

@article{Sager2013,
  author = {Sager, Sebastian},
  title = {Sampling Decisions in Optimum Experimental Design in the Light of Pontryagin's Maximum Principle},
  journal = {SIAM Journal on Control and Optimization},
  volume = {51},
  pages = {3181--3207},
  year = {2013},
  doi = {10.1137/110835098}
}

@book{PronzatoPazman2013,
  author = {Pronzato, Luc and P{\'a}zman, Andrej},
  title = {Design of Experiments in Nonlinear Models: Asymptotic Normality, Optimality Criteria and Small-Sample Properties},
  publisher = {Springer},
  year = {2013},
  doi = {10.1007/978-1-4614-6363-4}
}

@book{WalterPronzato1997,
  author = {Walter, Eric and Pronzato, Luc},
  title = {Identification of Parametric Models from Experimental Data},
  publisher = {Springer},
  year = {1997},
}

@article{Fisher1922,
  author = {Fisher, R. A.},
  title = {On the mathematical foundations of theoretical statistics},
  journal = {Philosophical Transactions of the Royal Society of London, Series A},
  volume = {222},
  pages = {309--368},
  year = {1922},
  doi = {10.1098/rsta.1922.0009}
}

@book{Atkinson2007,
  author = {Atkinson, A. C. and Donev, A. N. and Tobias, R. D.},
  title = {Optimum Experimental Designs, with SAS},
  publisher = {Oxford University Press},
  year = {2007},
  doi = {10.1093/oso/9780199296590.001.0001}
}

@book{pukelsheim2006optimal,
  author = {Pukelsheim, Friedrich},
  title = {Optimal Design of Experiments},
  publisher = {SIAM},
  year = {2006},
  doi = {10.1137/1.9780898719109}
}

@article{Kiefer2018,
  author = {Kiefer, J.},
  title = {Optimum Experimental Designs},
  journal = {Journal of the Royal Statistical Society: Series B (Methodological)},
  volume = {21},
  pages = {272--304},
  year = {1959},
  doi = {10.1111/j.2517-6161.1959.tb00338.x}
}

@book{coverthomas1991,
  author = {Cover, Thomas M. and Thomas, Joy A.},
  title = {Elements of Information Theory},
  publisher = {Wiley},
  year = {1991},
  doi = {10.1002/0471200611}
}

@article{chalonerverdinelli1995,
  author = {Chaloner, Kathryn and Verdinelli, Isabella},
  title = {Bayesian Experimental Design: A Review},
  journal = {Statistical Science},
  volume = {10},
  pages = {273--304},
  year = {1995},
  doi = {10.1214/ss/1177009939}
}

@article{Ryan2016,
  author = {Ryan, Elizabeth G. and Drovandi, Christopher C. and McGree, James M. and Pettitt, Anthony N.},
  title = {A Review of Modern Computational Algorithms for Bayesian Optimal Design},
  journal = {International Statistical Review},
  volume = {84},
  pages = {128--154},
  year = {2016},
  doi = {10.1111/insr.12107}
}

@article{huanmarzouk2013,
  author = {Huan, Xun and Marzouk, Youssef M.},
  title = {Simulation-Based Optimal Bayesian Experimental Design for Nonlinear Systems},
  journal = {Journal of Computational Physics},
  volume = {232},
  pages = {288--317},
  year = {2013},
  doi = {10.1016/j.jcp.2012.08.013}
}
